\documentclass[11pt]{amsart}
\usepackage{texmac,amscd}
\operators{Stab,Sel,Gal,Om}
\calsymbols{c}{C,K,O}
\bbsymbols{}{G}
\bbsymbols{b}{H}

\def\ctimes{{\scriptstyle\times}}

\def\omsub#1#2{\omega_{\scriptscriptstyle\vphantom{'}\!#1\!/\!#2}}
\def\omsubt#1#2{\omega_{\scriptscriptstyle\!#1\!/\!#2}}
\def\Alphasub#1#2{\alpha_{\scriptscriptstyle\vphantom{'}#1/\!#2}}
\def\csub#1#2{c_{\scriptscriptstyle\vphantom{'}\!#1\!/\!#2}}
\def\m{{\mathfrak m}}

\def\req#1#2{\>\>\>{\buildrel\ref{#2}\over#1}\>\>\>}
\let\liff\Longleftrightarrow
\let\ge\geqslant
\let\le\leqslant

\begin{document}

\title{Local invariants of isogenous elliptic curves}

\author{Tim and Vladimir Dokchitser}
\ufootnote{Accepted for publication in Trans. Amer. Math. Soc. on 18.08.2013\\[-7pt]}
\address{Dept of Mathematics, University Walk, Bristol BS8 1TW, United Kingdom}
\email{tim.dokchitser@bristol.ac.uk}
\address{Emmanuel College, Cambridge CB2 3AP, United Kingdom}
\email{v.dokchitser@dpmms.cam.ac.uk}

\begin{abstract}
We investigate how various invariants of elliptic curves, 
such as the discriminant, Kodaira type, Tamagawa number and real and 
complex periods, change under an isogeny of prime degree $p$. 
For elliptic curves over $l$-adic fields, the classification is almost complete 
(the exception is wild potentially supersingular reduction when $l=p$),
and is summarised in a table.
\end{abstract}

\llap{.\hskip 10cm} \vskip 3mm

\maketitle

\tableofcontents

%%%%%%%%%%%%%%%%%%%%%%%%%%%%%%%%%
\section{Introduction}
%%%%%%%%%%%%%%%%%%%%%%%%%%%%%%%%%
\label{s:intro}

We address the question how various invariants of elliptic curves, 
such as the discriminant, Kodaira type, Tamagawa number and real and 
complex periods, change under an isogeny. 
As every isogeny factors as a composition of 
endomorphisms and isogenies of prime degree,
throughout the paper we just consider a fixed isogeny
$$
  \phi: E \lar E'
$$
of prime degree $p$. 
The first result is a slight extension of a theorem of Coates 
(\cite{CoaLMS}, appendix), relating the discriminants 
$\Delta_E$ and $\Delta_{E'}$:

\begin{theorem}
\label{ideltaquo}
Let $\cK$ be a field of characteristic 0 and $\phi: E\to E'$ 
a \hbox{$p$-isogeny} of elliptic curves over $\cK$.
If $p>3$, then $\Delta_E^p/\Delta_{E'}$ is a 12th power in~$\cK$.
For $p=2,3$ this is a 3rd, respectively 4th power.
\end{theorem}

\noindent
We present another proof of Coates' result, exploiting the fact that 
for $p>3$ $\Delta(p\tau)^p/\Delta(\tau)$ is a 12th power of a modular form on $\Gamma_0(p)$
(\S\ref{s:delta},\S\ref{s:appmod}).
For one application, see \v Cesnavi\v cius' work on the parity conjecture 
\cite{Kes}~\S5. 

We then consider the standard invariants 
of elliptic curves
over local fields.
The following table summarises our results for the valuations of minimal
discriminants $\delta, \delta'$ of $E, E'$, 
their Tamagawa numbers $c, c'$, 
Kodaira types and the leading term $\frac{\phi^*\omega'}{\omega}$ 
of $\phi$ on the formal group (see \S1.1 for the notation):

\vskip 1cm

%%%%%%%%%%%%%%%%%%%%%%%%%%% TABLE %%%%%%%%%%%%%%%%%%%%%%%%%%%%%%%%%%%%

\def\r#1{$^{\eqref{#1}}$}
\def\r#1{}

%$$
%  \dagger = \left\{
%     \begin{array}{lll}
%     p & \text{if} & \ker\phi\subset \hat E(m_{F_\phi})\cr
%     1 & \text{if} & \ker\phi\not\subset \hat E(m_{F_\phi}) \cr
%     \end{array} 
%  \right.\qquad (F_\phi=F(\ker\phi)).
%$$
%
%$$
%  \ddagger = \left\{
%     \begin{array}{lll}
%     1 & \text{if} & \frac{\Delta}{\Delta'}\text{ is a norm in }F/K,\cr
%     \frac12 & \text{if} & \Delta'\text{ is a norm in }F/K\text{ and }\Delta\text{ is not},\cr
%     2 &\text{if} & \Delta\text{ is a norm in }F/K\text{ and }\Delta'\text{ is not}.\cr
%     \end{array} 
%  \right.
%$$
%
%$$
%  \ast = \left\{
%     \begin{array}{lll}
%     3        & \text{if} & \text{$E/K$ has non-trivial 3-torsion} \cr
%     \frac 13 & \text{if} & \text{$E/K$ has trivial 3-torsion} \cr
%     \end{array} 
%  \right.
%$$

\def\rh#1{\smash{\raise-7pt\hbox{$#1$}}}
\def\rhv#1#2{\smash{\raise-#1pt\hbox{#2}}}

\noindent\hskip-13.5mm
\begin{tabular}{@{\vrule width 1.2pt\ }l|c|c|c|c@{\ \vrule width 1.2pt}}
\noalign{\hrule height 1.2pt}
\quad\ \
Reduction type of $E/K$ & $\delta,\delta'$ & 
  $\vphantom{\int^{X^{X^x}}}\frac{\phi^*\omega'}{\omega}\>\>
  \raise-2pt\hbox{${\tbuildrel{\text{\tiny up to}}\over{\text{\tiny unit\ }}}$}
  $& $\frac{c}{c'}$ & 
  \raise-2pt\hbox{${\tbuildrel{\text{\tiny Kodaira}}\over{\text{\tiny types for $E,E'$\ }}}$}
  \\[4pt]
%&& ($l=p$) &\cr
\noalign{\hrule height 1.2pt}
%\hline
\vphantom{$\int^\int$}%
good ordinary       &  \rh{\delta=\delta'=0} & $1$ or $p$ ($\dagger$)\r{omegaord} & \rh{1} & \rh\IZ \cr
good supersingular  &   & ? & & \cr
\hline
\vphantom{$\int^\int$}%
split mult., $v(j)\!=\!pv(j')$ & $\delta=p\delta'$ & $1$\r{omegapotmult} & $p$ & $\In{pn},\>\In{n}$ \cr
split mult., $pv(j)\!=\!v(j')$ & $\delta'=p\delta$ & $p$\r{omegapotmult} & $\frac 1p$ & $\In{n},\>\In{pn}$ \cr
\hline
\rh{\text{nonsplit mult., }v(j)\!=\!pv(j')} & \rh{\delta=p\delta'} & \rh{1}\r{omegapotmult} & \smaller[2]$1$ if $p\ne 2$ or $2|\delta'$
  & \rh{\In{pn},\>\In{n}}\cr
&&&\smaller[2]2 otherwise\hbox to 2.2em{\hfill}&\\[4pt]
\rh{\text{nonsplit mult., }pv(j)\!=\!v(j')} & \rh{\delta'=p\delta} & \rh{p}\r{omegapotmult} & \smaller[2]$1$ if $p\ne 2$ or $2|\delta$
  & \rh{\In{n},\>\In{pn}}\cr
&&&\smaller[2]$\frac12$ otherwise\hbox to 2.2em{\hfill}&\cr
\hline   
\vphantom{$\int^\int$}%
additive pot. mult. &&&&\cr
  $\quad$\rh{v(j)\!=\!pv(j')} & \smaller[2]\rh{\delta'=\delta+\frac{p-1}p v(j)} & \rh{1}\r{omegapotmult} 
      & \ 1\ \ if $p\ge 3$ & \smash{$\InS{n},\,\InS{\rhv{-0.6}{$\scriptscriptstyle n+\frac{p-1}{p}v(j)$}}$} \cr
  &&& $(\ddagger)$ if $p=2$ & \rhv0{\smaller[4]$\>\>\qquad(=\InS{n\!/\!p}$ if $l\!\ne\!2$)} \\[7pt]
  $\quad$\rh{pv(j)\!=\!v(j')} & \smaller[2]\rh{\delta'\!=\!\delta\!-\!(p\!-\!1)v(j)} & $\rh{p}$\r{omegapotmult} 
      & \ 1\ \ if $p\ge 3$ & $\InS{n},\,\InS{\scriptscriptstyle n-(p-1)v(j)}$\cr
  &&& $(\ddagger)$ if $p=2$ & {\smaller[4]$\>\>\qquad(=\InS{pn}$ if $l\!\ne\!2$)}\cr
\hline
\vphantom{$\int^\int$}%
additive pot. good, $l\ne p$ &  &  &  & \cr
%  $\quad p>3$ &&& 1\cr
  $\quad p=3$, type \IV,\IVS, $\mu_3\!\not\subset\! K$ & 
    && $(\ast)$ & \cr
%  $\quad p=3$, all other cases &&& 1\cr
  $\quad p=2$, type \IZS & $\delta'=\delta$ & 1 & $(\ddagger)$ & \text{same}\cr
%  $\quad p=2$, all other cases &&& 1\cr
  $\quad$all other cases &&& 1 & \cr
%&&& $\bigleftchoice{3^{\pm1}\!\!\!\!\!\!}{\text{ type }\IV,\IVS, \mu_3\not\subset K}{1}{\text{otherwise}}$ if $p=3$\cr
%&&& $\bigleftchoice{2^{\pm1}\!\!\!\!\!\!}{\text{ type }\IZS, \mu_4\not\subset K}{1}{\text{otherwise}}$ if $p=2$\cr
 %  & $\delta/\delta'=p^{\pm 1}$ &&\cr
\hline
\vphantom{$\int^\int$}%
additive pot. good, $l=p$ &&&&\cr
\rh{\quad\text{pot. ordinary}} & \rh{\delta'=\delta} & \rh{1\>\text{or}\>p\ (\dagger)}& \ 1\ \ if $p\ge 3$ & \rh{\text{same}} \cr
  &&& $(\ddagger)$ if $p=2$ &\\[2pt]
$\quad$pot. supersingular tame & $\delta'=12-\delta$ & ? & 1 & \text{opposite}\cr
$\quad$pot. supersingular wild & ? & ? & ? & ? \\[5pt]
\noalign{\hrule height 1.2pt}
\multicolumn{5}{@{\vrule width 1.2pt\ }l@{\ \vrule width 1.2pt}}{%
\vphantom{$\int^{\int^\int}$}%
\smaller[2]
$K/\Q_l$ finite, $v\!: K^\times\!\surjects\!\Z$ valuation; \
$\phi\!: E/K\!\to\!E'/K$ $p$-isogeny; \ 
$\Delta, \Delta'$ minimal discriminants; \
}\cr
\multicolumn{5}{@{\vrule width 1.2pt\ }l@{\ \vrule width 1.2pt}}{\smaller[2]
$\delta=v(\Delta)$, $\delta'=v(\Delta')$; \
$\omega, \omega'$ minimal differentials; \
$j, j'$ $j$-invariants; \
$c, c'$ Tamagawa numbers.
}\cr
\hline
\multicolumn{5}{@{\vrule width 1.2pt\ }l@{\ \vrule width 1.2pt}}{%
\vphantom{$\int^\int$}%
\smaller[2]
$(\ast) = 3$ if $E/K$ has non-trivial 3-torsion, and $\frac 13$ otherwise.
}\\[3pt]
%\multicolumn{4}{|l|}{\smaller[3]
%  $\quad
%  (\ast) = \,\Bigl\{
%     \begin{array}{lll}\vphantom{\int^X}
%     3        & \text{if} & \text{$E/K$ has non-trivial 3-torsion} \cr
%     \frac 13 & \text{if} & \text{$E/K$ has trivial 3-torsion} \cr
%     \end{array} $
%}\\[4pt]
\multicolumn{5}{@{\vrule width 1.2pt\ }l@{\ \vrule width 1.2pt}}{\smaller[2]
%  \quad
  $(\ddagger) = \left\{
     \begin{array}{lll}
     1\!\! & \text{if}\!\! & \frac{\Delta}{\Delta'}\text{ is a norm in }F/K\cr
     \frac12 & \text{if}\!\! & \Delta'\text{ is a norm in }F/K\text{ and }\Delta\text{ is not}\cr
     2\!\! &\text{if}\!\! & \Delta\text{ is a norm in }F/K\text{ and }\Delta'\text{ is not}\cr
     \end{array} 
  \right.$\quad
  $\>\>(\dagger) = \Bigl\{
     \begin{array}{lll}
     p\!\! & \text{if}\!\! & \ker\phi\subset \hat E(\m_L)\cr
     1\!\! & \text{if}\!\! & \ker\phi\not\subset \hat E(\m_L) \cr
     \end{array} 
  \quad (L=F(\ker\phi))$
}\cr
\multicolumn{5}{@{\vrule width 1.2pt\ }l@{\ \vrule width 1.2pt}}{\smaller[2]
$F$ is any $(\dagger)$, respectively quadratic ($\ddagger$), extension 
  where $E$ has good or split multiplicative reduction.
}\cr
%\multicolumn{4}{|l|}{\smaller[3]
%  \quad
%  $(\ddagger) = \left\{
%     \begin{array}{lll}
%     1 & \text{if} & \frac{\Delta}{\Delta'}\text{ is a norm in }F/K,\cr
%     \frac12 & \text{if} & \Delta'\text{ is a norm in }F/K\text{ and }\Delta\text{ is not},\cr
%     2 &\text{if} & \Delta\text{ is a norm in }F/K\text{ and }\Delta'\text{ is not}.\cr
%     \end{array} 
%  \right.$
%}\cr
\noalign{\hrule height 1.2pt\phantom{$a_a$}}
\multicolumn{5}{c}{Table 1. Local invariants of isogenous elliptic curves}
\end{tabular}

%%%%%%%%%%%%%%%%%%%%%%%%% END TABLE %%%%%%%%%%%%%%%%%%%%%%%%%%%%%%%%%%

\smallskip

The quotient of Tamagawa numbers $\frac{c}{c'}$ and 
the quantity
$\frac{\phi^*\omega'}{\omega}$ 
classically appear in the applications of the isogeny invariance of the
Birch--Swinnerton-Dyer formula to Selmer groups of elliptic curves, see
e.g. \cite{Bir}, \cite{Scha}, \cite{Fis}, \cite{isogroot}, \cite{Keil}
and \cite{megasha}.
The quotient $\frac{\phi^*\omega'}{\omega}$ 
is an important invariant of the isogeny $\phi$, being the leading term of $\phi$
on the formal groups (cf. Lemma \ref{alphalead}, \cite{Sil1} \S IV.4
and also \cite{Scha} p.91, where it is denoted by $\phi'(0)$).

For curves defined over $\R$ and $\C$ 
the analogues of the local Tamagawa numbers are periods (cf. Remark \ref{remper})
$$
  \Omega(E,\omega) =
  {\int_{E(\R)} |\omega|}
\qquad\text{and}\qquad
  \Omega(E,\omega) =
  {2\int_{E(\C)} \omega\wedge\bar\omega}
$$
computed with respect to some invariant differential $\omega$ on $E$. 
We show that these periods for $E$ and $E'$ are related as follows:
%Such a period is related to the corresponding one $\Omega(E',\omega')$ on 
%the isogenous curve as follows:

\begin{theorem}
\label{iperiods}
Suppose the base field of $E, E'$ is $\cK=\R$ or $\cK=\C$. Choose  
invariant differentials $\omega, \omega'$ for $E$ and $E'$. Then
$$
  \frac{\Omega(E,\omega)}{\Omega(E',\omega')} =
   \lambda
   \Bigl|\frac{\omega}{\phi^*\omega'}\Bigr|_\cK.
$$
Here $|\cdot|_\cK$ is the standard normalised absolute value on $\cK$, and
$\lambda$
%=\frac{|\ker\phi: E(\cK)\to E'(\cK)|}{|\coker\phi: E(\cK)\to E'(\cK)|}$ 
is 
\begin{itemize}
\item $p$ if $\cK=\C$,
\item $p$ if $\cK=\R$, $p\ne 2$ and $\ker\phi\subset E(\R)$,
\item $1$ if $\cK=\R$, $p\ne 2$ and $\ker\phi\not\subset E(\R)$.
\end{itemize}
If $\cK\!=\!\R$ and $p\!=\!2$, write $E$ in the form $y^2\!=\!x^3\!+\!ax^2\!+\!bx$ 
so that $(0,0)\in\ker\phi$. Then $\lambda$ is 
\begin{itemize}
\item $1$ if $b>0$, and either $a<0$ or $4b>a^2$,
\item $2$ otherwise.
\end{itemize}
\end{theorem}

Finally, we look at periods of isogenous elliptic curves over $\Q$. 
In this case, $E/\Q$ and $E'/\Q$ have global minimal differentials 
$\omega, \omega'$, unique up to signs. The real periods 
$\Omega=\Omega(E/\R,\omega)$ and $\Omega'=\Omega(E'/\R,\omega')$
are the ones that enter the Birch--Swinnerton-Dyer conjecture over $\Q$. 
We prove that the quotient $\Omega/\Omega'$ is $1$, $p$ or $1/p$ 
and give a criterion for when it is $1$ (Theorem~\ref{perQpar}).
E.g., for $p\!>\!3$ the periods are equal if and only if $E$ has an odd number
of primes of additive reduction with local root number $-1$.
If $p>2$ and $E$ is semistable, then $\frac{\Omega}{\Omega'}=p^{\pm 1}$, and
$$
  \frac{\Omega}{\Omega'}=p
     \quad\liff\quad
  \omega=\pm \phi^* \omega'
     \quad\liff\quad
  \ker\phi\subset E(\Q),
$$
see Theorem \ref{OmegaSS}.

\medskip

%With the exception of this theorem, and of \S\ref{s:periods}-\S\ref{s:perQ} 
%which discuss real and complex periods, 
%zzthe rest of the paper is concerned with $l$-adic invariants. Thus in
%\S\ref{s:delta}-\S\ref{s:tam} 
%%(except in Theorem \ref{deltaquo}), 
%the curves $E, E'$ and the $p$-isogeny $\phi$ are defined over a finite extension 
%$K$ of $\Q_l$; $p=l$ is allowed. 
%%Most of the results are summarised in Table 1.

\par\noindent
{\bf 1.1. Notation.} 
Throughout the paper $p$ is a prime number, and $\phi: E\to E'$ an 
isogeny of elliptic curves of degree $p$. We write $\phi^t: E'\to E$ for the
dual isogeny. 
In \S\ref{s:disckod1}-\S\ref{s:tam}, the base field 
$K$ is a finite extension of $\Q_l$; $l=p$ is allowed. 

\newpage

\noindent
There we use the following notation:

\begin{tabular}{llll}
\vphantom{$\int^X$}%
$v$ & normalised valuation $K^\times\to \Z$\cr
$\m_K$ & maximal ideal of the ring of integers of $K$\cr
$\Delta, \Delta'$ & minimal discriminants of $E/K$ and $E'/K$\cr
$\delta, \delta'$ & their valuations: $\delta=v(\Delta), \delta'=v(\Delta')$\cr
$\omega, \omega'$ & minimal invariant differentials on $E, E'$ (N\'eron differentials),\cr
  &unique up to units\cr
$f=f'$ & conductor exponent of $E$ and $E'$\cr
$c, c'$ & local Tamagawa numbers of $E$ and $E'$\cr 
$m, m'$ & number of components in the special fibre of the minimal\cr
  &regular model of $E$, $E'$; so $\delta=f+m-1, \delta'=f'+m'-1$\cr
  &by Ogg's formula\cr
$j, j'$ & $j$-invariants of $E$ and $E'$\cr 
$\hat E, \hat E'$ & the formal groups of $E$ and $E'$ with respect to a minimal\cr
& Weierstrass equation.\cr
$|\cdot|_\cK$ & normalised absolute value. So 
$|x|_\R=|x|$,  $|x|_\C=|x|^2$ and \cr& 
 $|x|_K=q^{-v(x)}$ if $\cK=K$ is as above, with residue field $\F_q$.
\\[2pt]
\end{tabular}

\noindent
When we work over an extension $F/K$, we write 
$\Delta_{E/F}, \Delta_{E'/F}$ etc.
Any two invariant differentials on $E/K$ differ by a scalar, $\omega_1=a\omega_2$ 
with $a\in K^\times$, and we will abuse the notation slightly and write 
$\frac{\omega_1}{\omega_2}$ for $a$. 
%The minimal differential $\omega$ is unique up to a unit.
%Note also that $\Delta_F/\Delta=(\omsub{E}{F}/\omega)^{12}\cdot$unit for a finite
%extension $F/K$ (measures the discrepancy of the minimal models, \cite{?}).

Recall that a curve $E/K$ has additive reduction 
if and only if it has conductor exponent $f\ge 2$, 
and $f=2$ if and only if the $\ell$-adic Tate module of $E$ is tamely ramified
for some (any) $\ell\ne l$.
We will call this {\em tame} reduction (and {\em wild} otherwise). 
If $l\ge 5$, the reduction is always tame; when $l=2$ it is tame if and only if 
$E$ has Kodaira type $\IV, \IVS$; when $l=3$ it is tame if and only if
$E$ has Kodaira type $\III, \IIIS$ or $\IZS$ (cf. Theorem \ref{deltaisog}).
In Table 1, {\em opposite} Kodaira types refers to 
$\II\leftrightarrow\IIS$, $\III\leftrightarrow\IIIS$, $\IV\leftrightarrow\IVS$,
 $\IZS\leftrightarrow\IZS$.

\par\smallskip\noindent
{\bf 1.2. Layout.}
Theorem \ref{ideltaquo} is proved in \S\ref{s:delta}.
In \S\ref{s:disckod1}-\S\ref{s:tam} we prove the results summarised in 
Table 1: for $\delta$ see \ref{delmain} and \ref{delpotord5};
%? Good reduction not done anywhere.
for $\phi^*\omega'/\omega$ see
\ref{omegaord}, \ref{omegapotmult}, \ref{lnep:omega} and \ref{omegapotord};
for the Tamagawa numbers, see \ref{tammain};
for Kodaira symbols, see \ref{lnep:kod}.
Real and complex periods are discussed in \S\ref{s:periods} and the particular
case of elliptic curves over $\Q$ in \S\ref{s:perQ}.
Appendix \ref{s:apptate} recalls the theory of the Tate curve and 
some standard facts about quadratic twists. 
Appendix \ref{s:appmod} reviews the connection between values of 
modular forms and invariants of elliptic curves with a cyclic isogeny.

%? $V_l E$, conductor exponent, root number, $L$-function (local and global),
%  rank/Selmer rank (if over number field) are all unchanged under isogeny.

\begin{acknowledgements}
We would like to thank Anthony Scholl, William Hart and John Coates 
for discussions related to modular forms and the $\Delta(E')=\Delta(E)^p$ 
result, and the referee for helpful comments.
The first author is supported by a Royal Society
University Research Fellowship.
\end{acknowledgements}

\newpage

%%%%%%%%%%%%%%%%%%%%%%%%%%%%%%%%%
\section{$\Delta(E')=\Delta(E)^p$ up to 12th powers}
%%%%%%%%%%%%%%%%%%%%%%%%%%%%%%%%%
\label{s:delta}

In this section we relate the discriminants $\Delta$ and $\Delta'$ 
of $p$-isogenous elliptic curves $E$ and $E'$.
Here we work over an arbitrary field of characteristic~0, 
so these are discriminants
of some (not necessarily minimal) Weierstrass models. They depend on 
the choice of models, and are well-defined up to 12th powers.

\begin{theorem}[Coates \cite{CoaLMS}, appendix, Thm. 8]
\label{deltaquo}
Let $\cK$ be a field of characteristic 0 and $\phi: E\to E'$ 
a $p$-iso\-geny of elliptic curves over $\cK$ with
$p\!>\!3$. Then $\Delta^p/\Delta'$ is a 12th power~in~$\cK$.
\end{theorem}

\begin{proof}
We may assume that $\cK\subset \C$, 
embedding the field of definition of $E$, $E'$ and $\phi$ into $\C$ if 
necessary (Lefschetz principle).

Let $\tau$ be a complex variable in the upper-half plane, and
let $\eta(\tau)$ be the Dedekind eta-function.
By a classical theorem, $\eta(p\tau)^p/\eta(\tau)$ is a mo-\linebreak dular form 
of weight $\frac{p-1}2$ on $\Gamma_0(p)$, with character $(\frac dp)$, and 
its square\linebreak
$f(\tau)=[\eta(p\tau)^p/\eta(\tau)]^2$ is a modular form 
of weight $p-1$ on $\Gamma_0(p)$, with trivial character
(\cite{Klya} Thm~2.2, or \cite{New} Thm 1.1 and \cite{GS} remark below Thm.~1). 
Its $q$-expansion
$$
  f(\tau) = 
    q^{\frac{p^2-1}{12}}
    \prod_{n\ge 1}\frac{(1-q^{pn})^{2p}}
         {(1-q^n)^2} \qquad\qquad (q=e^{2\pi i \tau}),
$$
clearly has integer coefficients. Note that $f(\tau)$ is a 12th root of
$\frac{\Delta(p\tau)^p}{\Delta(\tau)}$.

Choose models for $E/\cK, E'/\cK$ of the form
$$
  E: y^2=4x^3+ax+b, \qquad \qquad  E': y^2=4x^3+a'x+b',
$$
with $\phi^*\frac{dx}y=p\frac{dx}y$, and complex uniformisations 
$E\!=\!\C/\Lambda$, $E'\!=\!\C/\Lambda'$ so that
$$
  \varphi: \C/\Lambda\ni z\mapsto (\wp_\Lambda(z),\wp'_\Lambda(z))\in E(\C)
$$
satisfies $\varphi^*\frac{dx}y=dz$ and similarly for~$\Lambda'$
(cf. Appendix \ref{s:appmod}). 
%
%$\Lambda=\Z\Omega_1\!+\Z\Omega_2$,
%$\Lambda'=\Z\Omega_1+\Z p\,\Omega_2$ and $\phi(z)=pz\mod\Lambda'$,
%as 
%In other words, we have 
%with $\Omega_1, \Omega_2$ normalised so that 
%satisfies $
From $\phi^* dz=p dz$ we see that $\Lambda\subset\Lambda'$ has index $p$,
and so we can write the two lattices in the form
$\Lambda=\Z\Omega_1\!+\Z\Omega_2$,
$\Lambda'=\Z\Omega_1+\Z p\,\Omega_2$.

By the $q$-expansion principle (or Theorem \ref{modeval}),
$$
  \bigl(\tfrac{2\pi}{\Omega_1}\bigr)^{p-1} f(\tau) 
     \in \cK, \qquad \tau=\tfrac{\Omega_2}{\Omega_1}.
$$
On the other hand, 
$$
  f(\tau)^{12} = \frac{\Delta(p\tau)^p}{\Delta(\tau)} = 
    \frac
    {\bigl(\frac{\Omega_1}{2\pi}\bigr)^{\!12p} \Delta_{E'}^p}
    {\bigl(\frac{\Omega_1}{2\pi}\bigr)^{\!12} \Delta_E};
$$
here $\Delta_E$ and $\Delta_{E'}$ are the discriminants of the models
$E: y^2=x^3+\frac a4x+\frac b4$ and $E': y^2=x^3+\frac {a'}4x+\frac {b'}4$,
and the second equality follows from the relation between $\Delta(\tau)$ 
and $\Delta_E$ proved in \eqref{deldel}.
%two 
%specific $\cK$-rational models, 
%cf.~\eqref{deldel}. 
It follows that 
$$
  \frac{\Delta_{E'}^p}{\Delta_E} = [\bigl(\tfrac{2\pi}{\Omega_1}\bigr)^{p-1} f(\tau)]^{12}
    \in \cK^{\times 12}.
$$
Swapping $E$ and $E'$ (or using the fact that $(\Delta^p)^p$ is $\Delta$ up 
to a 12th power, as $p^2\equiv 1\mod 12$ for $p\ne 2,3$) gives the claim.
\end{proof}

Now we prove analogues for $p=2$ and $p=3$, using
an explicit computation with a universal family:

\begin{theorem}
\label{deltaquo2}
Let $\cK$ be a field of characteristic 0 and $\phi: E\to E'$ 
a $2$-iso\-geny of elliptic curves over $\cK$. Then
$\Delta^2/\Delta'$ is a 3rd power in~$\cK$.
\end{theorem}

\begin{proof}
Any 2-isogeny $\phi: E\to E'$ of elliptic curves
over a field of characteristic not 2 or 3 has a model
\beq
  E&:& y^2=x^3+ax^2+bx, \cr
  E'&:& y^2=x^3-2ax^2+(a^2-4b)x,\cr
  \phi(x,y)&=&(x+a+bx^{-1}, y-byx^{-2}),
\eeq
and
$$
  \frac{\Delta^2}{\Delta'} = \frac{[16b^2(a^2-4b)]^2}{256b(a^2-4b)^2} = b^3.
$$
\end{proof}

\begin{theorem}
\label{deltaquo3}
Let $\cK$ be a field of characteristic 0 and $\phi: E\to E'$ 
a $3$-iso\-geny of elliptic curves over $\cK$. Then
$\Delta^3/\Delta'$ is a 4th power in~$\cK$.
\end{theorem}

\begin{proof}
Any 3-isogeny $\phi: E\to E'$ of elliptic curves
%\pagebreak[3]
over a field of characteristic not 2 or 3 has a model
%? explain: the x-coordinate of the point in the kernel is rational,
\beq
  E&:& y^2=x^3+a(x-b)^2, \cr
  E'&:& y^2=x^3+ax^2+18abx+ab(16a-27b), \cr
  \phi(x,y)&=&(x\!-\!4abx^{-1}\!+\!4ab^2x^{-3},y\!+\!4abyx^{-2}\!-\!8ab^2yx^{-3}),
\eeq
and
$$
  \frac{\Delta^3}{\Delta'} = \frac{[-16a^2b^3(4a+27b)]^3}{-16a^2b(4a+27b)^3} = (4ab^2)^4.
$$
%Conversely, for any $a, b$ in the ground field the formulae do define a
%$p$-isogeny, provided the resulting curves are non-singular
%(equivalently if $a(a^2\!-\!4b)\ne 0$ for $p=2$, and if
%$ab(4a+27b)\ne0$ for $p=3$).
\end{proof}

%%%%%%%%%%%%%%%%%%%%%%%%%%%%%%%%%
\section{Discriminants and Kodaira types I}
%%%%%%%%%%%%%%%%%%%%%%%%%%%%%%%%%
\label{s:disckod1}

Throughout \S\ref{s:disckod1}-\S\ref{s:tam}
we follow the notation of \S1.1. In particular, 
$K$ is a finite extension of $\Q_l$, and $\phi: E/K\to E'/K$ is a $p$-isogeny.

\begin{theorem}
\label{deltaisog}
Suppose $E/K$ has additive potentially good reduction. 
Then $E$ has tame reduction (equivalently, has conductor exponent 2) if and only~if 
\begin{itemize}
\item $l\ge 5$, or
\item $l=3$ and $E$ has Kodaira type \III, \IIIS, \IZS, or
\item $l=2$ and $E$ has Kodaira type \IV, \IVS.
\end{itemize}
In this case $E'$ is tame as well, and
$$
  0<\delta,\delta'<12 
    \qquad\text{and}\qquad
  \delta' \equiv p\,\delta \mod 12.
$$
\end{theorem}

\begin{proof}
For the first statement, see \cite{Sil2} IV.9, Table 4.1 for $l\ge 5$, 
\cite{Kra} Thm.~1 for $l=3$, and 
\cite{KT} Prop 8.20 for $l=2$. From \cite{Sil2} IV.9, Table 4.1 it also follows
that $0<\delta,\delta'<12$ in all these cases. The last congruence follows
from Theorems \ref{deltaquo}--\ref{deltaquo3}.
\end{proof}

%? Make self-contained?
\begin{theorem}
\label{potordkod}
Suppose $E/K$ has additive potentially good reduction and is not 
a quadratic twist of a curve with good reduction. Then
\begin{itemize}
\item If $l=2,3$ or $l\equiv -1\mod 12$ then $E$ is potentially supersingular.
\item If $l\equiv 1\mod 12$, then $E$ is potentially ordinary.
\item If $l\equiv 5\mod 12$, then $E$ is potentially ordinary if and only if
its Kodaira type is $\III$ or $\IIIS$.
\item If $l\equiv 7\mod 12$, then $E$ is potentially ordinary if and only if
its Kodaira type is $\II, \IIS, \IV$ or $\IVS$.
\end{itemize}
\end{theorem}

\begin{proof}
Let $K^{nr}$ be the maximal unramified extension of $K$, and
$F/K^{nr}$ the (finite) extension cut out by the Galois action
on any $\ell$-adic Tate module of $E$ for $\ell\ne l$. 
By the criterion of N\'eron-Ogg-Shafarevich, 
$F$ is the unique minimal Galois extension of $K^{nr}$ 
where $E$ has good reduction.

The Galois group $\Gal(F/K^{nr})$ has order at least 3, since $E$ is not 
a quadratic twist of a curve with good reduction (cf. Lemma \ref{quadtwist}).
%? little exercise
As explained in \cite{ST} proof of Thm. 2, it acts 
faithfully on the reduced curve $\tilde E$ defined over the residue field of $F$
as a group of automorphisms. This forces $j(\tilde E)$ to be either 0 or 1728,
see e.g. \cite{Sil1} Thm. III.10.1. If $l=2$ or 3, then $1728=0$ is a supersingular 
$j$-invariant, as asserted.

Now suppose $l>3$. 
By \cite{Sil2} IV.9, Table 4.1, either 
a) $E$ has reduction type $\II, \IIS, \IV$ or $\IVS$
and $j(E)$ reduces to 0, or b) $E$ has reduction type $\III, \IIIS$ and $j(E)$
reduces to 1728.
The $j$-invariant 0 is ordinary if and only
if $l\equiv 1\mod 3$, and 1728 is ordinary if and only
if $l\equiv 1\mod 4$; see \cite{Sil1} Ex. V.4.4,~V.4.5.
%\begin{itemize}
%\item
%$\delta=2$, $\delta'=10$, $p\equiv 2\mod 3$, $j\equiv 0\mod \m_K$.
%\item
%$\delta=4$, $\delta'=8$, $p\equiv 2\mod 3$, $j\equiv 0\mod \m_K$.
%\item
%$\delta=3$, $\delta'=9$, $p\equiv 3\mod 4$, $j\equiv 1728\mod \m_K$.
%\end{itemize}
\end{proof}

\begin{corollary}
\label{delpotord5}
Suppose $l=p$ and $E$ has additive tame potentially good reduction. 
If the reduction is potentially ordinary, 
then $\delta=\delta'$ and $E, E'$ have the same Kodaira type. 
If the reduction is potentially supersingular, then 
$\delta=12-\delta'$ and $E, E'$ have opposite Kodaira types 
$(\II\leftrightarrow\IIS$, $\III\leftrightarrow\IIIS$, $\IV\leftrightarrow\IVS$,
 $\IZS\leftrightarrow\IZS)$.
\end{corollary}

\begin{proof}
%If $p=3$, by Lemma \ref{potkodord} either $E, E'$ have good reduction 
%($\delta=\delta'=0$, Kodaira type $\In0$)
%or they are quadratic twists of elliptic curves with good reduction
%($\delta=\delta'=6$, Kodaira type $\IZS$; see \cite{Kra}).
%? little exercise, or see Kob? 
%
%Suppose $p\ne 3$.
By Theorem \ref{deltaisog}, we have $\delta, \delta'<12$. Also,
if $\delta\!\ne\!\delta'$, then $\delta\!\not\equiv\!p\delta\!\mod\!12$,
equivalently $12\nmid\delta(p\!-\!1)$. 
Exchanging $E$, $E'$ if necessary, the possibilities with $\delta\!\ne\!\delta'$ 
are (cf. \cite{Sil2} IV.9, Table 4.1)
\begin{itemize}
\item
$\delta=2$, $\delta'=10$, $p\equiv 2\mod 3$,
\item
$\delta=4$, $\delta'=8$, $p\equiv 2\mod 3$,
\item
$\delta=3$, $\delta'=9$, $p\equiv 3\mod 4$.
\end{itemize}
By Theorem \ref{potordkod}, these are precisely the cases of 
potentially supersingular reduction unless $E$ is quadratic twist 
of a curve with good reduction. In the latter case, $l$ cannot be 2 
(as $E$ has tame reduction), so $E$ and $E'$ have 
Kodaira type $\IZS$ and $\delta=\delta'=6$.
\end{proof}

%%%%%%%%%%%%%%%%%%%%%%%%%%%%%%%%%
\section{Differentials}
%%%%%%%%%%%%%%%%%%%%%%%%%%%%%%%%%
\label{s:disc}

\begin{notation}
\label{notalpha}
We will write
$$
  \Alphasub{\phi}{K} = \Bigl|\frac{\phi^*\omega'}{\omega}\Bigr|_K^{-1}.
$$
%where $|\cdot|_K$ is the normalised absolute value on $K$. 
\end{notation}

\begin{lemma}
\label{alphalead}
\par\noindent
\begin{enumerate}
\item
The isogeny $\phi$ induces a map on formal groups,
%$\phi^*\omega'/\omega$ is the leading term on formal groups. 
$$
  \phi: \hat E(\m_K) \to \hat E'(\m_K), \qquad \phi(T) = a T + \cdots,
$$  
with leading term $a=\frac{\phi^*\omega'}{\omega}\ctimes$unit$\ \>\in \cO_K$.
\item
$$
   \frac{|\coker\phi: E(K)\to E'(K)|}{|\ker\phi: E(K)\to E'(K)|} = 
      \Alphasub{\phi}{K}\,\frac{c'}{c}.
$$
\end{enumerate}
\end{lemma}

\begin{proof}
(1) By the N\'eron universal property, $\phi$ 
extends to a morphism of N\'eron models, and thus induces
a map on formal groups. For the leading term, 
see \cite{Sil1} Ch. IV, especially Cor. IV.4.3. 
% formula for w(z)=1+O(z) on p 113 shows that minimal differential on E
% is the normalised minimal differential on the formal group.
(2) \cite{Scha} Lemma 3.8.
\end{proof}

\begin{lemma}
\label{lnep:omega}
If $l\ne p$, then $\phi^*\omega'$ is minimal, so $\Alphasub{\phi}{K}=1$.
\end{lemma}

\begin{proof}
%Write $\phi^*\omega'=a\omega$. Then $a\in O_K$ since $\phi^*\omega'$ is integral.
%Now $a(\phi^t)^*\omega=(\phi^t)^*\phi^*\omega'=p\omega'$ is a minimal
%differential as $p$ is a unit. So $a$ is a unit and $\phi^*\omega'$ is minimal.
%
Write $\phi^*\omega'=a\omega$, $(\phi^t)^*\omega=a'\omega'$
with $a, a'\in\cO_K$ by Lemma \ref{alphalead}.
%Since the pullback of an integral differential is integral, we must have 
%$a, a'\!\in\!\cO_K$. 
Because $\phi^t\phi=[p]$, we have $aa'=p\in\cO_K^\times$. 
So $a$ and $a'$ are units, and $\phi^*\omega'$ is minimal.
%
%Then $a\in O_K$ since $\phi^*\omega'$ is integral.
%Now $a(\phi^t)^*\omega=(\phi^t)^*\phi^*\omega'=p\omega'$ is a minimal
%differential as $p$ is a unit. So $a$ is a unit and $\phi^*\omega'$ is minimal.
%
%
%
%The minimal differential $\omega'$ on $E'/K$ is integral,
%and therefore so is $\tilde\omega=\phi^*\omega'$ on $E$. 
%Also, $(\phi^t)^*\tilde\omega=m\omega'$ is minimal. So $\tilde\omega$ must be minimal 
%as well. (If it were not, we would have $\tilde\omega=a \omega$ for some $a$
%in the maximal ideal, and $\omega'=a\cdot$ integral differential
%cannot be minimal.)
\end{proof}

\begin{lemma}
\label{delquo}
Suppose $F/K$ is a finite extension. Then 
$$
  \frac{\phi^*\omsub{E'\!}{K}}{\omsub{E}{K}} = \frac{\phi^*\omsub{E'\!}{F}}{\omsub{E}{F}}
  \times \text{unit}
  \qquad \Longleftrightarrow \qquad 
  \frac{\Delta_{E/K}}{\Delta_{E'/K}} = 
  \frac{\Delta_{E/F}}{\Delta_{E'/F}} \times \text{unit}.
$$
If $l\!\ne\!p$, or $E/K$ is semistable, or $l=p$ and $E$ has tame potentially 
ordinary reduction, then the formulae~hold.
\end{lemma}

\begin{proof}
It is easy to see that up to units (cf. \cite{Sil1} Table III.1.2), 
$$
  \frac{\Delta_{E/K}}{\Delta_{E/F}} = \Bigl(\frac{\omsub{E}{K}}{\omsub{E}{F}}\Bigr)^{-12}
    \qquad\text{and}\qquad
  \frac{\Delta_{E'/K}}{\Delta_{E'/F}} 
    = \Bigl(\frac{\omsub{E'\!}{K}}{\omsub{E'\!}{F}}\Bigr)^{-12}
    = \Bigl(\frac{\phi^*\omsub{E'\!}{K}}{\phi^*\omsub{E'\!}{F}}\Bigr)^{-12}.
$$
So
$\frac{\Delta_{E/K}}{\Delta_{E'/K}}/\frac{\Delta_{E/F}}{\Delta_{E'/F}}$
is the 12th power of 
$\frac{\phi^*\omsub{E'\!}{F}}{\omsub{E}{F}} / \frac{\phi^*\omsub{E'\!}{K}}{\omsub{E}{K}}$, up
to a unit.

For the second claim, if $l\ne p$ or $E/K$ is semistable, then the left-hand
formula holds (Lemma \ref{lnep:omega} and the fact that for semistable curves 
minimal differentials stay minimal in all extensions).
If $l=p$ and $E$ is tame, the right-hand
formula holds by Corollary \ref{delpotord5}.
%Write $\omsub{E}{K}=a\omsub{E}{F}$ and $\omsub{E'\!}{K}=a'\omsub{E'\!}{F}$, with $a,a'\in\cO_F$.
%Then
%$$
%  \omsub{E}{K} = a\omsub{E}{F} = a\phi^* \omsub{E'\!}{F} \cdot\text{unit}
%     = \frac{a}{a'}\, \phi^* \omsub{E'\!}{K} \cdot\text{unit}.
%%     = \frac{a}{a'}\, \omsub{E}{K} \cdot\text{unit},
%$$
%%and so $a$ and $a'$ have the same valuation.
%On the other hand,
%$$
%  \frac{\Delta_{K}}{\Delta_{F}} = a^{12} \cdot\text{unit},
%  \qquad
%  \frac{\Delta'_{K}}{\Delta'_{F}} = (a')^{12} \cdot\text{unit},
%$$
%and the claim follows.
%%
\end{proof}

\begin{remark}
\label{twistisog}
Suppose $E$ and $E'$ are in Weierstrass form,
$$
  E:y^2=f(x), \qquad \qquad E': y^2=g(x).
$$
Since $\phi(-P)=-\phi(P)$ and every even rational function on $E$ is 
a function of $x$ (cf. \cite{Sil1}, proof of Cor. III.2.3.1), 
$\phi$ has the form 
$$
  \phi: (x,y) \longmapsto (\xi(x),y\eta(x)), \qquad \xi(x),\eta(x)\in K(x).
$$
If $F=K(\sqrt d)$ is a quadratic extension, and 
$$
  E_d:dy^2=f(x), \qquad \qquad E'_d: dy^2=g(x)
$$
the quadratic twists of $E,E'$ by $d$, then the same 
formula $(\xi(x),y\eta(x))$ defines an isogeny $\phi_d: E_d\to E'_d$.
%$E_d$ 
%admits the corresponding isogeny
%$$
%  \phi_d: (x,y) \longmapsto (\xi(x),dy\eta(x)).
%$$
It fits into a commutative diagram
$$
\begin{CD}
  0 @>>> E_d(K) @>>> E(F) @>N>> E(K) @>>> \frac{E(K)}{N E(F)} @>>> 0 \\
    @.  @V{\phi_d}VV @V{\phi}VV @V{\phi}VV @V{\phi}VV \\
  0 @>>> E'_d(K) @>>> E'(F) @>N>> E'(K) @>>> \frac{E'(K)}{N E'(F)} @>>> 0, \\
\end{CD}
$$
where 
%$\phi_d$ is the isogeny that agrees with $\phi$ over $F$, 
the map
$E_d(K)\to E(F)$ is $(x,y)\mapsto (x,y\sqrt d)$, and 
$N$ is the norm (or trace) map $E(F)\to E(K)$, $E'(F)\to E'(K)$.
\end{remark}

\begin{lemma}
\label{kercoker4}
Let $F=K(\sqrt d)$ be a quadratic extension,
$E_d, E'_d$ the quadratic twists of $E,E'$ by $d$, and $\phi_d$ the
corresponding isogeny. The groups $\frac{E(K)}{N E(F)}$, 
$\frac{E'(K)}{N E'(F)}$ are finite, and
$$
   \Bigl(\Alphasub{\phi_d}{K}\frac{\csub{E'_d}{K}}{\csub{E_d}{K}}\Bigr)^{-1}\cdot
   \Alphasub{\phi}{F}\frac{\csub{E'\!}{F}}{\csub{E}{F}}\cdot
   \Bigl(\Alphasub{\phi}{K}\frac{\csub{E'\!}{K}}{\csub{E}{K}}\Bigr)^{-1}\cdot
   \frac{|\frac{E'(K)}{N E'(F)}|}{|\frac{E(K)}{N E(F)}|} = 1.
$$
\end{lemma}

\begin{proof}
The groups $\frac{E(K)}{N E(F)}$, $\frac{E'(K)}{N E'(F)}$ are 
quotients of $\frac{E(K)}{2E(K)}$, $\frac{E'(K)}{2E'(K)}$, which are 
%\newpage\noindent
finite. 
Now consider the commutative diagram above. 
Because the alternating product of $|\ker|/|\coker|$ is 1,
Lemma \ref{alphalead}(2) gives the claim.
\end{proof}

\begin{proposition}
\label{twistmain}
Let $F=K(\sqrt d)$ be a quadratic extension,
$E_d, E'_d$ the quadratic twists of $E,E'$ by $d$, and $\phi_d$ the
corresponding isogeny. 
%Write 
%$$
%  \Alphasub{\phi}{K}  = \Bigl|\frac{\phi^*\omega'}{\omega}\Bigr|_K^{-1}
%$$
%and similarly $\Alphasub{\phi}{F}$ for $E/F$ and $\Alphasub{\phi_d}{K}$ for $E_d/K$.

\noindent
(1) Write $K_n, F_n$ for the degree $n$ unramified extensions of $K, F$.
%and $N$ for the norm (or trace) maps $E(F_n)\to E(K_n)$ and $E'(F_n)\to E'(K_n)$.
Then
$$
  \frac{\Alphasub{\phi}{K}\Alphasub{\phi_d}{K}}{\Alphasub{\phi}{F}} =
  \arrowlim{\text{$\vphantom{\int^a}n$ odd}}{\scriptstyle n\to\infty}\sqrt[n]{
    \tfrac{
    \>\>|{E'(K_n)}/{N E'(F_n)}|\>\>
    }{
    |{E(K_n)}/{N E(F_n)}|
    }}.
$$
If $l\ne 2$, this quotient is 1.

\noindent
(2) 
We have $\Alphasub{\phi_d}{K}=\Alphasub{\phi}{K}$ and $\Alphasub{\phi}{F}=\Alphasub{\phi}{K}^2$
unless (i) $l=p$ and $E$ has additive potentially supersingular reduction
or (ii) $l=p=2$ and $E$ has supersingular reduction. 
\end{proposition}

\begin{proof}
(1)
We 
%have a commutative diagram
%$$
%\begin{CD}
%  0 @>>> E_d(K) @>>> E(F) @>N>> E(K) @>>> \frac{E(K)}{N E(F)} @>>> 0 \\
%    @.  @V{\phi_d}VV @V{\phi}VV @V{\phi}VV @V{\phi}VV \\
%  0 @>>> E'_d(K) @>>> E'(F) @>N>> E'(K) @>>> \frac{E'(K)}{N E'(F)} @>>> 0, \\
%\end{CD}
%$$
%Now replace $K$ and $F$ by $K_n$ and $F_n$. 
apply Lemma \ref{kercoker4} for $E$ in $F_n/K_n$; because $n$ is odd, we 
have $F_n=K_n(\sqrt d)$. 
The minimal differentials stay
the same in unramified extensions, so $\Alphasub{\phi}{K_n}=\Alphasub{\phi}{K}^n$,
and similarly for $\Alphasub{\phi}{F_n}$ and $\Alphasub{\phi_d}{K_n}$. Thus,
$$
%\begin{equation}\label{KTlimit}
\frac{\csub{E'\!}{F_n}}{\csub{E}{F_n}}
\frac{\csub{E}{K_n}}{\csub{E'\!}{K_n}}
\frac{\csub{E_d}{K_n}}{\csub{E'_d}{K_n}}
    \frac{
    \>\>|{E'(K_n)}/{N E'(F_n)}|\>\>
    }{
    |{E(K_n)}/{N E(F_n)}|
    }
%  \frac{|\frac{E'(K)}{N E'(F)}|}{|\frac{E(K)}{N E(F)}|}
  =\frac{\Alphasub{\phi}{K}^n\Alphasub{\phi_d}{K}^n}{\Alphasub{\phi}{F}^n}.
%\end{equation}
$$
All the Tamagawa numbers are bounded, so the claim follows by taking 
$n$th roots and letting $n\to\infty$.
Moreover, if $l\ne 2$, because the norm quotients are 2-groups and $\alpha$'s are 
powers of $l$, we can compare the $l$-parts before taking the limit, and we 
find that $\frac{\Alphasub{\phi}{K}\Alphasub{\phi_d}{K}}{\Alphasub{\phi}{F}}=1$.

(2) We may assume $l=p$, as otherwise all $\alpha$'s are 1 by 
Lemma \ref{lnep:omega}. 

(a) 
%The norm quotients in the above formula are 2-groups and $\alpha$'s are 
%powers of $l$. Therefore, if $l\ne 2$, we can take the $l$-parts 
%and let $n\to\infty$, and we find that $\frac{\Alphasub{\phi}{K}\Alphasub{\phi_d}{K}}{\Alphasub{\phi}{F}}=1$
%in this case. 
If $l\ne 2$, or $E$ has either good ordinary or 
split multiplicative reduction, 
we have
${\Alphasub{\phi}{F}}={\Alphasub{\phi}{K}\Alphasub{\phi_d}{K}}$:
if $l\ne 2$, this is proved in (1); otherwise, 
the norm quotients have size at most 4 
by \cite{KT} Prop~8.6, Prop~4.1, and so
${\Alphasub{\phi}{F}}={\Alphasub{\phi}{K}\Alphasub{\phi_d}{K}}$ by (1).

(b) 
If either $E$ is semistable or
$l=p>3$ and $E$ is potentially ordinary, 
%We claim that $\Alphasub{\phi}{F}=\Alphasub{\phi}{K}^2$ 
$$
  \frac{\phi^*\omsub{E'\!}{K}}{\omsub{E}{K}} = \frac{\phi^*\omsub{E'\!}{F}}{\omsub{E}{F}} \times \text{unit},
$$ 
by Lemma \ref{delquo}, and 
%In the former case, the minimal 
%models and minimal differentials of $E$ and $E'$ over $K$ remain 
%minimal over $F$, so
$$
%\daggerequation{samecomp}{$\ast$}{
%\begin{equation}
%\def\theequation{$\ast$}
%\label{samecomp}
\Alphasub{\phi}{F} = \Bigl|\frac{\phi^*\omsub{E'\!}{F}}{\omsub{E}{F}}\Bigr|_F^{-1} =
  \Bigl|\frac{\phi^*\omsub{E'\!}{K}}{\omsub{E}{K}}\Bigr|_K^{-2} = \Alphasub{\phi}{K}^2.
%\end{equation}
%}
%$\phi^*\omega'/\omega$ stays the same, and its normalised absolute value
%gets squared. 
$$
%In the latter case, 
%$\frac{\phi^*\omsub{E'\!}{K}}{\omsub{E}{K}} = \frac{\phi^*\omsub{E'\!}{F}}{\omsub{E}{F}}
%  \times \text{unit}$ by Lemma \ref{delquo}, 
%$$
%  \Bigl(\frac{\omsub{E}{K}}{\omsub{E}{F}}\Bigr)^{12} =
%  \frac{\Delta_K}{\Delta_F} \>\> {\buildrel\eqref{delpotord5}\over=}\>\> 
%  \frac{\Delta'_K}{\Delta'_F} =
%  \Bigl(\frac{\omsub{E'\!}{K}}{\omsub{E'\!}{F}}\Bigr)^{12} 
%  = \Bigl(\frac{\phi^*\omsub{E'\!}{K}}{\phi^*\omsub{E'\!}{F}}\Bigr)^{12} 
%$$
%up to units, and the same computation \eqref{samecomp} works.

%the valuation of $\phi^*\omega'/\omega$ 
%again stays the same in $F/K$, and the claim follows.

(c) Combining (a) and (b), we find that 
$\Alphasub{\phi}{K}=\Alphasub{\phi_d}{K}$ 
and $\Alphasub{\phi}{F}=\Alphasub{\phi}{K}^2$ 
in the following three cases:
\begin{itemize}
\item 
$l>3$ and $E$ is semistable or potentially ordinary, 
\item
$l=3$ and $E$ is semistable,
\item
$l=2$ and $E$ is split multiplicative or good ordinary.
\end{itemize}
It follows that $\Alphasub{\phi}{K}=\Alphasub{\phi_d}{K}$ 
and $\Alphasub{\phi}{F}=\Alphasub{\phi}{K}^2$ 
also hold for quadratic twists of all such curves, 
as $\Alphasub{\phi_{d_1}}{K}=\Alphasub{\phi}{K}=\Alphasub{\phi_{d_2}}{K}$ for any pair of twists. 
Since a curve with potentially multiplicative reduction is a 
quadratic twist of a semistable one, and a curve with potentially ordinary
reduction a quadratic twist of a good ordinary one when $l\le 3$
(Theorem \ref{potordkod}), the result holds in all the cases claimed.
\end{proof}

%
%\begin{remark}
%The analogue of
%$$
%  \frac{\Delta_{K,\min}}{\Delta'_{K,\min}} = 
%  \frac{\Delta_{F,\min}}{\Delta'_{F,\min}} \times \text{unit}.
%$$
%for exterior forms also holds for abelian varieties.
%\end{remark}
%

\begin{proposition}
\label{omegaord}
Suppose $l=p$ and $E$ has good ordinary reduction. 
Let $F=K(\ker\phi)$ be the field obtained by adjoining
the coordinates of points in $\ker\phi$. 
Then
$$
  \frac{\phi^*\omega'}{\omega} = 
    \bigleftchoice{\text{$p\ctimes$unit}}{\text{if }\ker\phi\subset \hat E(\m_F)}
                  {\text{unit}}{\text{otherwise}}.
$$
\end{proposition}

\begin{proof}
%When $p=2$, this is proved in \cite{isogroot} \S7.6-\S7.7. Suppose $p$
The isogeny $\phi$ induces an isogeny on formal groups 
$$
  \phi: \hat E(\m_F) \to \hat E'(\m_F), \qquad \phi(T) = a T + \cdots,
$$  
with $a=\frac{\phi^*\omega'}{\omega}$ by Lemma \ref{alphalead} (1).
Define $a'$ similarly for $\phi^t$.
The reduction $\tilde E = E \mod \m_F$ is an ordinary elliptic curve, 
so $[p]=\tilde\phi\circ\tilde\phi^t$ is an isogeny of height 1 on its formal
group. Hence either $\tilde\phi$ or $\tilde\phi^t$ is an isomorphism on formal
groups of the reduced curves, in other words either $a\mod \m_F$ or $a'\mod \m_F$ is non-zero.
Because $aa'=p$, one of $a, a'$ is a unit and the other one is $p\ctimes$unit.
If $a$ is a unit, then $\ker\phi$ is trivial on $\hat E$. 
Otherwise, $\phi$ reduces to an inseparable isogeny of prime degree, and hence
$\ker\tilde\phi=0$ on~$\tilde E$. Therefore $\ker\phi$ lies on the formal group.
%$\ker\phi=\ker[p]$ is of size $p$ on $\hat E$.
%See \cite{isogroot} \S6 for $p$ odd and \S7.6-\S7.7 for $p=2$.
\end{proof}

\begin{proposition}
\label{omegapotord}
%Let $K$ be a $p$-adic field and $\phi: E\to E'$ a $p$-isogeny of elliptic 
%curves over $K$. 
If $l=p$ and $E$ has potentially ordinary reduction,
then $\frac{\phi^*\omega'}{\omega}$ is either
a unit or $p\ctimes$unit. If $F/K$ is finite, then 
$\frac{\phi^*\omsub{E'\!}{F}}{\omsub{E}{F}}=\frac{\phi^*\omsub{E'\!}{K}}{\omsub{E}{K}}\,\ctimes\,$unit.
\end{proposition}

\begin{proof}
When $p=2$ or $3$, Theorem \ref{potordkod} shows that $E$ is a quadratic twist
of a curve with good reduction. 
The result follows from Propositions \ref{omegaord} and~\ref{twistmain}(2).

When $p\ge 5$, Lemma \ref{delquo} shows that 
%let $F$ be any extension of $K$ where $E$ has good reduction. 
%By Lemma \ref{delpotord5} the curves $E$ and $E'$ have the
%same minimal discriminants over $K$ (and trivially over $F$),
$\frac{\phi^*\omsub{E'\!}{K}}{\omsub{E}{K}} = \frac{\phi^*\omsub{E'\!}{F}}{\omsub{E}{F}}$
for any $F/K$. Taking $F$ to be the field where $E$ acquires good reduction,
we see that this quantity
%so 
%by Lemma \ref{delquo}. 
%The right-hand side 
is a unit or $p\ctimes$unit by
Proposition \ref{omegaord}.
%
%By Corollary \ref{corkodisog}, the Kodaira symbols of $E$ and $E'$ are the
%same in the potentially ordinary case (little exercise; separate corollary?).
%Now apply Proposition \ref{omegaord} and the argument in Lemma \ref{lnep:delta}.
%
%
%Case 2. $p=2$. First do good ordinary reduction (prove $\Alphasub{\phi}{K}=1$ in some
%explicit direction.) 
%Next, consider the diagram of Proposition \ref{omegapotmult}, with $E_s=$
%good ordinary quadratic twist of $E$. Again consider kernels/cokernels. 
%The contribution from the middle two columns is the same by the good ordinary
%case. The last columm is bounded after an unramified quadratic extension 
%(actually trivial from \cite{KT} Prop 8.6). Now the first column gives
%the norm of $\Alphasub{\phi}{K}$ up to Tamagawa numbers, so the norm of $\Alphasub{\phi}{K}$ is bounded
%as we go up the unramified tower. Hence $\Alphasub{\phi}{K}$ is a unit. Now apply 
%Lemma \ref{delquo}. [Move this to the differential section. I.e. upgrade 
%Proposition \ref{omegapotord} to include $p=2$ and 3.]
%
%Case 3. $p=3$. See Silverman's table.
\end{proof}

\begin{proposition}
\label{omegapotmult}
If $E$ has potentially multiplicative reduction, then
%Let $K$ be an $l$-adic field and $\phi: E\to E'$ a $p$-isogeny of elliptic 
%curves over $K$. Suppose $E$ has non-integral $j$-invariant. Then
$$
  \frac{\phi^*\omega'}{\omega} = 
    \bigleftchoice{\text{unit}}{\text{if }v(j)=p\,v(j')}{\text{$p\,\ctimes$unit}}{\text{otherwise}}.
$$
In particular, $\frac{\phi^*\omsub{E'\!}{K}}{\omsub{E}{K}} = \frac{\phi^*\omsub{E'\!}{F}}{\omsub{E}{F}}\,\ctimes\,$unit
for every finite extension $F/K$.
%Take a quadratic twist $E_t/K$ and a canonically induced isogeny 
%$\phi_t: E_t\to E'_t$ (same kernel). Then
%%$$
%%  \frac{\Delta_{E,\min}}{\Delta_{E',\min}} = 
%%  \frac{\Delta_{E_t,\min}}{\Delta_{E_t',\min}}.
%%$$
%$$
%  \frac{\omsub{E,\min}}{\phi^*\omsub{E',\min}} = 
%  \text{unit}\times
%  \frac{\omsub{E_t,\min}}{\phi_t^*\omsub{E_t',\min}}.
%$$
\end{proposition}

\begin{proof} 
Note that if $l\ne p$, the result follows from Lemma \ref{lnep:omega}.
Because both the $j$-invariant and $\alpha$ are unchanged under quadratic twists 
(Proposition~\ref{twistmain}), we may assume
that $E$ has split multiplicative reduction. 

By the theory of the Tate curve (Theorem \ref{tatecurve}), 
the pair $E, E'$ is $E^{(q^p)}, E^{(q)}$
(in some order), with $q\in \m_K$. In particular, either
$v(j)=p\,v(j')$ or $v(j')=p\,v(j)$. Because
$$
  \frac{\phi^*\omega'}{\omega} 
  \frac{(\phi^t)^*\omega}{\omega'}  = 
  \frac{\phi^*\omega'}{\omega} 
  \frac{\phi^*(\phi^t)^*\omega}{\phi^*\omega'}  =  
%  \frac{p\,\omega}{\omega} = 
  \frac{\phi^*\omega'}{\omega} \frac{p\,\omega} {\phi^*\omega'} =
  p,
$$
the claim for $\phi$ is equivalent to that for $\phi^t$. Swapping $E$ and $E'$
if necessary, assume that $E=E^{(q^p)}, E'=E^{(q)}$, 
in which case $\phi$ is given by
$$
  \phi: E(K)=K^\times/(q^p)^\Z \lar K^\times/q^\Z=E'(K),
$$
induced by the identity map on $K^\times$. Here $|\ker\phi|=p$,
$|\coker\phi|=1$ on $E(K)$, and $\frac{c}{c'}=\frac{v(q^p)}{v(q)}=p$. 
By Lemma \ref{alphalead}, the quotient $\frac{\phi^*\omega'}{\omega}$
is a unit.%, as asserted.
\end{proof}

%\begin{remark}
%The proof also shows that the formula of Proposition ?
%$$
%  \frac{\phi^*\omega'}{\omega} = 
%    \bigleftchoice{\text{$p\cdot$unit}}{\text{if }\ker\phi\subset \hat E(\m_F)}
%                  {\text{unit}}{\text{otherwise}}.
%$$
%\end{remark}

%%%%%%%%%%%%%%%%%%%%%%%%%%%%%%%%%
\section{Discriminants and Kodaira types II}
%%%%%%%%%%%%%%%%%%%%%%%%%%%%%%%%%
\label{s:disckod2}

\begin{theorem}
\label{delmain}
\par\noindent
%Let $E\to E'$ be a $p$-isogeny of elliptic curves over $K$.
%Write $\delta, \delta'$ for the valuations of their minimal discriminants.
%If $l=p$, assume that $E$ does not have additive potentially supersingular 
%reduction.
\begin{enumerate}
\item
If $E$ has potentially good reduction, and either $l\ne p$ or the
reduction is good or potentially ordinary, then $\delta=\delta'$.
\item
If $E$ has multiplicative reduction, 
then $\frac{\delta}{\delta'}=\frac{v(j)}{v(j')}=p^{\pm 1}$.
\item
If $E$ has potentially multiplicative reduction, 
then 
$$
  \delta-\delta'=v(j')-v(j)=
    \bigleftchoice {\frac{1-p}{p}v(j)}{\text{if }v(j)=pv(j')}{(p-1)v(j)}{\text{if }v(j')=pv(j)}.
$$
\end{enumerate}
\end{theorem}

\begin{proof}
If $E$ has good reduction, then $\delta=\delta'=0$. 
If $E$ has split multiplicative reduction, then 
$E$ and $E'$ are Tate curves with parameters $q$ and $q^p$, in some order
(Theorem \ref{tatecurve}). 
So $\delta=-v(j)$ and $\delta'=-v(j')$ are $v(q)$ and $pv(q)$, in some order.
%The valuations of their discriminants and 
%$j$-invariants are $-v(q)$ and $-v(q^p)$.
Thus (1) holds in the good reduction case, and (2), (3) in the 
split multiplicative case.

\vskip -1mm

If either $l\ne p$ or $E$ is potentially multiplicative,
then $\frac{\phi^*\omsubt{E'\!}{K}}{\omsubt{E}{K}}=\frac{\phi^*\omsubt{E'\!}{F}}{\omsubt{E}{F}}$ for
any $F/K$, by Lemma \ref{lnep:omega} and Proposition \ref{omegapotmult}.
Taking $F$ to be a field where $E$ has good or split multiplicative reduction,
we find that the claim for $E/F$ implies that for $E/K$, by Lemma \ref{delquo}.

We are left with the case that $E$ has additive potentially ordinary 
reduction with $l=p$. If $p>3$, the claim is proved in Corollary \ref{delpotord5}.
If $p=2, 3$, then $E/K$ is a quadratic twist of a curve $E_d/K$ with good ordinary reduction
(Theorem \ref{potordkod}). 
Let $F=K(\sqrt d)$ be the corresponding quadratic extension. 
In the notation of Proposition \ref{twistmain} we have
$\Alphasub{\phi}{K}=\Alphasub{\phi_d}{K}$ and,
since the minimal model of $E_d$ stays minimal in $F/K$ and $E/F\iso E_d/F$,
also $\Alphasub{\phi}{F}=\Alphasub{\phi_d}{K}^2$. So $\Alphasub{\phi}{F}=\Alphasub{\phi}{K}^2$.
In other words, $\frac{\phi^*\omsubt{E'\!}{K}}{\omsubt{E}{K}}=\frac{\phi^*\omsubt{E'\!}{F}}{\omsubt{E}{F}}$,
and the
\linebreak\par\vskip-4mm\noindent
claim follows, again by Lemma \ref{delquo}.
%
%
%
%In all other cases, let $F$ be a field where $E$ has either good or 
%split multiplicative reduction. Swapping $E$ and $E'$ if necessary, we may
%arrange that $\phi^*\omsub{E'\!}{K}=\omsub{E}{K}$ and $\phi^*\omsub{E'\!}{F}=\omsub{E}{F}$
%(Lemma \ref{lnep:omega} when $l\ne p$ and Proposition \ref{omegapotmult} in
%the potentially multiplicative case).
%% 
%%
%%We will show that for a finite extension $F/K$,
%$$
%  \frac{\Delta_{K,\min}}{\Delta'_{K,\min}} = 
%  \frac{\Delta_{F,\min}}{\Delta'_{F,\min}} \times \text{unit}.
%$$
%Now apply Lemma \ref{delquo}.
\end{proof}

\begin{remark}
Note that in the potentially good case, the formulae 
$\delta=\delta'$ (Theorem \ref{delmain}) 
and $\delta'\equiv p\delta\mod 12$ (Theorem \ref{deltaisog}) 
do not contradict each other. The reason is that the possible reduction
types are restricted in the potentially ordinary case, see Theorem \ref{potordkod}.
%
%contradict \ref{deltaisog}. (Because $E$ acquires good reduction
%after a cyclic extension of degree dividing $p-1$; this restricts possible
%reduction types (Theorem \ref{potordkod}).
\end{remark}

\begin{remark}
In the potentially supersingular case 
the formulae in Proposition \ref{omegapotord} 
and Theorem \ref{delmain}(1) may not hold.
For example, consider the 5-isogenous elliptic curves $E=50b1$, $E'=50b3$.
%! (Fisher5 with $\lambda=-2$). 
Their reduction types over $\Q_5$ are $\II$ and $\IIS$ respectively,
so $\delta\ne\delta'$. Also, over $\Q_5(\sqrt 5)$ the reduction types become 
$\IV$ and $\IVS$, and $\phi^*\omega'/\omega=\sqrt{5}$
(computed as in Lemma \ref{delquo} from the minimal discriminants), 
which is neither a unit nor $5\ctimes$unit.
\end{remark}

%\pagebreak

\begin{theorem}
\label{lnep:kod}
%Suppose $p\ne l$.
\par\noindent
\begin{enumerate}
\item
If $E$ has potentially good reduction and $p\ne l$, then
the Kodaira types of $E$ and $E'$ are the same.
\item 
If $E$ has potentially good ordinary reduction and $p=l$, then 
the Kodaira types of $E$ and $E'$ are the same.
\item
If $E$ has tame potentially good supersingular reduction and $p=l$, then 
$E$ and $E'$ have opposite Kodaira type
$(\II\leftrightarrow\IIS$, $\III\leftrightarrow\IIIS$, $\IV\leftrightarrow\IVS$,
 $\IZS\leftrightarrow\IZS)$.
%\item
%If $E$ has potentially good tame reduction (conductor exponent 2) and $p=l$, then 
%\begin{enumerate}
%\item if $E$ is potentially ordinary, then
%$E$ and $E'$ have the same Kodaira type.
%\item if $E$ is potentially supersingular, then
%$E$ and $E'$ have opposite Kodaira type
%$(\II\leftrightarrow\IIS$, $\III\leftrightarrow\IIIS$, $\IV\leftrightarrow\IVS$,
% $\IZS\leftrightarrow\IZS)$.
%\end{enumerate}
\item
If $E$ has multiplicative reduction, then the Kodaira 
type is $\In{n}$ for~$E$, and either $\In{pn}$ or $\In{n/p}$ for $E'$,
corresponding to $v(j')\!=\!pv(j)$ and $pv(j')\!=\!v(j)$.
\item
If $E$ has additive potentially multiplicative reduction,  
then the Kodaira type is $\InS{n}$ for $E$ and $\InS{n'}$ for $E'$,
where either
$$
  \qquad\qquad
  v(j')\!=\!pv(j)  
    \quad\>\>\text{and}\quad\>\>
  n'\!=\!pn \!-\! 4(a\!-\!1)(p\!-\!1)  \!=\! n\!-\!v(j)(p\!-\!1)
$$
or vice versa (swap $n\leftrightarrow n', j\leftrightarrow j'$); here $a$ 
is the conductor exponent of the quadratic character of $K(\sqrt{-c_6})/K$
(it is $1$ if $l\ne 2$), where $c_6$ is the standard invariant 
of $E$ as in {\rm\cite{Sil1} \S III.1}. 
\end{enumerate}
\end{theorem}

\begin{proof}
Write $f, f'$ for the conductor exponents of $E$ and $E'$, and 
$m, m'$ for the number of connected components of the special fibre of their
minimal regular models. Because $\phi$ induces an isomorphism 
between the $\ell$-adic Tate modules of $E$ and $E'$ for $\ell\ne l,p$,
we have $f=f'$. Recall that $\delta=f+m-1$ and $\delta'=f'+m'-1$ by 
Ogg's formula \cite{Sil2} IV.11.1.

(1) By Theorem \ref{delmain}, $\delta=\delta'$ and so $m=m'$.
If $l\ne 2$, then from the reduction type table  
\cite{Sil2} IV.9, Table 4.1 we see
that the Kodaira type in the additive potentially good case is determined by $m$,
so they are the same for $E$ and $E'$.
(Note that $\InS n$ is necessarily potentially 
multiplicative even when $l=3$.)

Suppose $l=2$. Then $m$ almost determines the reduction type,
except for the pairs $\{\InS 2,\IVS\}$, $\{\InS 3,\IIIS\}$ and $\{\InS 4,\IIS\}$.
Passing to the maximal unramified extension if necessary, we see that 
$\InS 0$ and $\InS n$ are the only reduction types with 
(2-part of) the local Tamagawa number equal to 4. 
Since the 2-part of the Tamagawa number is invariant under $\phi$
(as $p=\deg\phi$ is odd, $\phi$ induces an isomorphism between the 2-parts
of $E/E_0$ and $E'/E'_0$),
the Kodaira types must be the same.
%So it suffices to prove that the 2-part of . (Note that $p$ is odd in this case.)

(2) In the tame case, this is Corollary \ref{delpotord5}.
By Theorem \ref{potordkod}, the only wild case is 
when $p=l=2$ and $E, E'$ are quadratic twists of curves with good 
ordinary reduction by some character $\chi$. 
Here $\delta=\delta'$ by Theorem \ref{delmain}(1), 
so $m=m'$ as in (1). Again as in (1), the Kodaira types of $E$ and $E'$ are 
the same, except possibly 
for 3 pairs of cases $\{\InS 2,\IVS\}$, $\{\InS 3,\IIIS\}$ and $\{\InS 4,\IIS\}$.
We claim that none of these can occur. For the first one, the reduction
type $\IVS$ is tame by Theorem \ref{deltaisog}. For the second one,
$m=8$ and $6|\delta$ (since $E$ acquires good reduction after a quadratic
extension), so $f$ is odd by Ogg's formula; however $f$ is twice the 
conductor exponent of $\chi$, contradiction. In the last case, 
%? actually II* cannot even occur for quadratic twists of good ordinary curves, 
%  because II* seems to become III* after a ramified quadratic extension,
%  in residue characteristic 2.
%  See Tate's algorithm Step 8, where the quadratic polynomial y^2-unit 
%  in the residue field always has a double root. (See also twostar.m)
pass to the maximal unramified extension as in (1). Then the Tamagawa 
numbers of $E$ and $E'$ become 1 and 4 (\cite{Sil2} IV.9, Table 4.1),
but their quotient is $1,2$ or $\frac 12$ by the very last case 
of Theorem \ref{tammain}. (The proof of this case does not use the present
theorem.)

(3) This is a special case of Corollary \ref{delpotord5}.

%When $l=p\ge 5$, see Corollary \ref{delpotord5}. For $l=p\le 3$, use
%\ref{deltaisog} and the fact that the discriminants of $E$ and $E'$
%determine the Kodaira type in the tame case (cf. \cite{Sil2} IV.9, Table 4.1).

(4) This follows from the theory of the Tate curve (Theorem \ref{tatecurve}).

(5) 
The quadratic twists $E_{-c_6}, E'_{-c_6}$ have split 
multiplicative reduction and are $p$-isogenous 
(Lemma \ref{quadmult}, Remark \ref{twistisog}).
If $v(j')=pv(j)$, then these twists have Kodaira types
$\In{\nu}$, $\In{p\nu}$ with $-\nu=v(j_{E_{-c_6}})=v(j)$ (Theorem \ref{tatecurve}).
By Theorem \ref{lorthm}, $E$ and $E'$ have Kodaira types $\InS{n}$ and $\InS{n'}$
with $n=\nu+4a-4$ and $n'=p\nu+4a-4$. Clearly
$$
  n' = p(n-4a+4)+4a-4 = pn - 4(p-1)(a-1).
$$
Because $v(j)=-\nu=-n+4a-4$, also
$$
  n' = pn - (p-1)(4a-4) = pn-(p-1)(v(j)+n) = n - (p-1)v(j).
$$
If, on the other hand, $v(j)\!=\!pv(j')$, swap $E$ and~$E'$. 
%
%By Corollary \ref{lorcor},
%$$
%  
%$$
%
%
%
%By Corollary \ref{lorcor},
%$$
%  \delta = \nu + 6 a = (n-4a+4) + 6a = n + 2a + 4,
%$$
%so $4a=2\delta-2n-8$ and
%$$
%  n' = pn-(p-1)(4a-4) = pn-(p-1)(2\delta-2n-12)
%$$
\end{proof}

%If $l\ne 2$, $m=5+n$, $\delta=6+n$, $\delta'=6p+np$, $m'=6p+np-1$, 
%$n'=np+6(p-1)$. 
%If $l=2$, 
%$m=5+n$, $f=2a$, $\delta=2a+5+n-1$, 
%$\delta'=2ap+np+4p$, $m'=2ap+np+4p + 1 - 2a$, and so 
%$n'=m'-5= 2ap+np+4p + 1 - 2a - 5=2a(p-1) -4 + (n+4)p
%= 2a(p-1) -4 + (n+4)(p-1) + (n+4) = n + (2a+n+4)(p-1) 
%= pn + (2a+4)(p-1)  = n+\delta(p-1)$.

%\begin{corollary}[of \ref{deltaisog}]
%\label{corkodisog}
%Let $K$ be a finite extension of $\Qp$, $p>3$.
%If $E$ has potentially good reduction, the Kodaira types of $E, E'$ can be 
%$\{\II,\II\}, \{\IV,\IV\}$ if $p\equiv 1 \mod 3$,
%$\{\II,\IIS\}, \{\IV,\IVS\}$ if $p\equiv 2 \mod 3$, 
%$\{\III,\III\}$ if $p\equiv 1 \mod 4$,
%$\{\III,\IIIS\}$ if $p\equiv 3 \mod 4$, or $\{\IZS,\IZS\}$,
%\end{corollary}

%%%%%%%%%%%%%%%%%%%%%%%%%%%%%%%%%
\section{Tamagawa numbers}
%%%%%%%%%%%%%%%%%%%%%%%%%%%%%%%%%
\label{s:tam}

%If $E$ is a quadratic twist of a semistable curve, we let
%$F/K$ be a quadratic extension such that $E/F$ is good or split multiplicative,
%and define
%$$
%  \ddagger = \left\{
%     \begin{array}{lll}
%     1 & \text{if} & \frac{\Delta}{\Delta'}\text{ is a norm in }F/K,\cr
%     \frac12 & \text{if} & \Delta'\text{ is a norm in }F/K\text{ and }\Delta\text{ is not},\cr
%     2 &\text{if} & \Delta\text{ is a norm in }F/K\text{ and }\Delta'\text{ is not}.\cr
%     \end{array} 
%  \right.
%$$

\begin{theorem}
\label{tammain}
If $E$ is semistable, then the ratio of Tamagawa numbers $\frac{c}{c'}$~is 
\begin{itemize}
\item $1$ if $E$ has good reduction.
\item $\frac{\delta}{\delta'}=\frac{v_K(j(E))}{v_K(j(E'))}=p^{\pm 1}$ 
if $E$ has split multiplicative reduction.
\item if $E$ has nonsplit multiplicative reduction: 
\begin{itemize}
\item[$\bullet$] $1$ if $p\ne 2$, or if both $\delta$ and $\delta'$ are even,
\item[$\bullet$] $2$ if $p=2$ and $\delta'$ is odd,
\item[$\bullet$] $\frac12$ if $p=2$ and $\delta$ is odd.
\end{itemize}
%
%$1$ if $p$ is odd;
%if $p=2$, then 
%$2^{\pm 1}$ if one of $\delta$, $\delta'$ is odd, otherwise $1$.
\end{itemize}
If $E$ has additive reduction and $p>3$, then $c/c'=1$.\\
If $E$ has additive reduction and $p=3$, then $c/c'$ is
\begin{itemize}
\item $1$ if $l\ne 3$, unless $E$ has type $\IV, \IVS$ and 
  $\mu_3\not\subset K$. In this exceptional case, 
\begin{itemize}
\item[$\bullet$] $3$ if $E(K)[3]\ne 0$,
\item[$\bullet$] $\frac 13$ if $E(K)[3]=0$.
\end{itemize}
\item $1$ if $l=3$ and $E$ 
has Kodaira type $\III,\IIIS,\IZS,\InS n$ (equivalently, 
$E$ does not have wild potentially supersingular reduction).
\end{itemize}
If $E$ has additive reduction and $p=2$, then $c/c'$ is
\begin{itemize}
\item $1$ if $l\ne 2$ and $E$ is not of type $\IZS$ or $\InS n$.
\item $1$ if $l\!=\!2$ and $E$ has tame potentially good reduction 
  (i.e. type~$\IV,\IVS$\!).
\item if a) $l\ne 2$ and $E$ has type $\IZS$ or $\InS n$, or b)
$l=2$ and $E$ does not have potentially 
supersingular reduction, 
\begin{itemize}
\item[$\bullet$] $1$ if $\frac{\Delta}{\Delta'}$ is a norm in $F/K$,
\item[$\bullet$] $\frac12$ if $\Delta'$ is a norm in $F/K$ and $\Delta$ is not,
\item[$\bullet$] $2$ if $\Delta$ is a norm in $F/K$ and $\Delta'$ is not,
\end{itemize}
where $F/K$ is a quadratic extension such that $E/F$ has good 
or split multiplicative reduction.
%\item $l\ne 2$: $\ddagger$ if $E$ has type $\IZS$ or $\InS n$, 
%  otherwise $1$.
%\item $1$ if $l=2$ and $E$ has tame potentially good reduction 
%  (i.e. type $\IV, \IVS$).
%\item $\ddagger$ if $l=2$ and $E$ does not have potentially 
%supersingular reduction.
%take $F/K$ quadratic such that $E/F$ is good or split multiplicative.
%Then 
%$c/c'=1$ if $(\frac{\Delta}{\Delta'},F/K)=-1$, 
%$c/c'=2$ if $(\Delta,F/K)=1=-(\Delta',F/K)$, and
%$c/c'=\frac12$ if $(\Delta,F/K)=-1=-(\Delta',F/K)$.
%$$
%  \frac{c}{c'} = \left\{
%     \begin{array}{lll}
%     1 & \text{if} & \frac{\Delta}{\Delta'}\text{ is a norm in }F/K,\cr
%     2 & \text{if} & \Delta'\text{ is a norm in }F/K\text{ and }\Delta\text{ is not},\cr
%     \frac12 &\text{if} & \Delta\text{ is a norm in }F/K\text{ and }\Delta'\text{ is not}.\cr
%     \end{array} 
%  \right.
%$$
\end{itemize}
\end{theorem}

\begin{lemma}
\label{tamprimetop}
The quotient $c/c'$ is a power of $p$.
\end{lemma}

\begin{proof}
The isogeny $\phi$ induces maps 
$E(K)\to E'(K)$ and $E_0(K)\to E'_0(K)$, and so $E/E_0\to E'/E'_0$. 
%? The fact that there is an induced map E_0(K)\to E'_0(K) seems to be 
%  a non-trivial fact. It does follow from the existence of the N\'eron model.
These are finite groups, and since $\phi\phi^t=[p]=\phi^t\phi$ are automorphisms
on their prime-to-$p$ parts, $\phi$ is an isomorphism between these 
prime-to-$p$ parts. 
%So the prime-to-$p$ parts of the Tamagawa numbers are the same.
\end{proof}

\begin{proof}[Proof of Theorem \ref{tammain}]
In the semistable case this follows from 
Tate's algorithm \cite{Sil2} IV.9
%the classification of Tamagawa
%numbers \cite{Sil1} 
%\cite{Sil2} IV.9 Table 4.1 
and Theorem \ref{delmain}. 
%Assume henceforth that $E$ and $E'$ have additive reduction,
%in particular $1\le c,c'\le 4$
%(cf. \cite{Sil1} VII.6.1, \cite{Sil2} IV.9 Table 4.1).
%In particular, if $p>3$, we are done by Lemma \ref{tamprimetop}.
Assume henceforth that $E$ and $E'$ have additive reduction.
In particular $1\le c,c'\le 4$
(cf. \cite{Sil1} VII.6.1, \cite{Sil2} IV.9 Table 4.1),
so for $p>3$ the result follows by Lemma \ref{tamprimetop}.

For $p=3$, $l\ne 3$ see \cite{isogroot} Lemma 11. 

Suppose $p=2$, $l\ne 2$. If the Kodaira type is not $\IZS$ or $\InS n$,
the Kodaira types of $E$ and $E'$ are the same by Theorem \ref{lnep:kod}. 
By the reduction type table (\cite{Sil2} IV.9, Table 4.1)
in the case $\II, \IIS, \IV, \IVS$ the 2-parts of 
the Tamagawa numbers are trivial, and in the case $\III, \IIIS$ they are both 2.
Hence $c$ and $c'$ have the same $2$-part, and are therefore equal.
When the reduction type is $\IZS$, see the computation in \cite{isogroot} \S7.4.

For $p=l=3$ and type $\III, \IIIS, \IZS$, the isogenous curve $E'$ also has
one of these three Kodaira types (Theorem \ref{deltaisog}).
%The Tamagawa numbers of $E$ and $E'$
%can be 1, 2 or 4, so they are unchanged under a 3-isogeny by Lemma \ref{tamprimetop}.
The Tamagawa numbers for~these 
%
%\pagebreak\noindent
types can be 1, 2 or 4,
so the 3-isogeny forces the equality $c=c'$ (Lemma \ref{tamprimetop}).
Similarly, for $p=l=2$ and type $\IV, \IVS$, the Tamagawa numbers are
1 or 3, and are unchanged by a 2-isogeny.

Finally, in the three remaining cases
($p=l=3$, type $\InS n$; 
$p=l=2$, $E$ not potentially supersingular;
or $p=2$, $l\ne 2$, type $\InS n$), 
some quadratic twist $E_d/K$ of $E/K$
has either good or split multiplicative reduction 
(Theorem \ref{potordkod}, Lemma \ref{quadmult}). 
%? Plus the fact that In* for l<>2 means potentially multiplicative.
Let $F=K(\sqrt d)$ be the corresponding quadratic extension.
By Lemma \ref{kercoker4} (see Notation \ref{notalpha} and Remark \ref{twistisog}
for the notation),
%As in the proof of Proposition \ref{twistmain}, 
%we consider the commutative diagram
%$$
%\begin{CD}
%  0 @>>> E(K) @>>> E_d(F) @>N>> E_d(K) @>>> \frac{E_d(K)}{N E_d(F)} @>>> 0 \\
%    @.  @V{\phi}VV @V{\phi_d}VV @V{\phi_d}VV @V{\phi_d}VV \\
%  0 @>>> E'(K) @>>> E'_d(F) @>N>> E'_d(K) @>>> \frac{E'_d(K)}{N E'_d(F)} @>>> 0, \\
%\end{CD}
%$$
%which leads to 
%
%where 
%%$\phi_d$ is the isogeny that agrees with $\phi$ over $F$, 
%the map
%$E_d(K)\to E(F)$ is $(x,y)\mapsto (x,y\sqrt d)$. 
%The groups $\frac{E(K)}{N E(F)}$, $\frac{E'(K)}{N E'(F)}$ are finite,
%being quotients of $E(K)/2E(K)$ and $E'(K)/2E'(K)$ which are finite. 
%Because $|\ker|/|\coker|$ is multiplicative in long exact sequences, 
%Lemma \ref{alphalead} gives
$$
   \Bigl(\Alphasub{\phi}{K}\frac{\csub{E'\!}{K}}{\csub{E}{K}}\Bigr)^{-1}\cdot
   \Alphasub{\phi}{F}\frac{\csub{E'_d}{F}}{\csub{E_d}{F}}\cdot
   \Bigl(\Alphasub{\phi_d}{K}\frac{\csub{E'_d}{K}}{\csub{E_d}{K}}\Bigr)^{-1}\cdot
   \frac{|\frac{E'_d(K)}{N E'_d(F)}|}{|\frac{E_d(K)}{N E_d(F)}|} = 1.
$$
%$\Alphasub{\phi}{K}=|\frac{\phi^*\omega'}{\omega}|_K^{-1}$ and 
%$\Alphasub{\phi_d}{K}$ and $\Alphasub{\phi}{F}$ are similarly defined for $E_d/K$ and $E_d/F\iso E/F$;
%see Remark \ref{twistisog} for the definition of $\phi_d$.
Proposition \ref{twistmain}(2) shows that $\Alphasub{\phi}{K}\!=\!\Alphasub{\phi_d}{K}$,
and $\Alphasub{\phi}{F}\!=\!\Alphasub{\phi}{K}^2\!=\!\Alphasub{\phi}{K}\Alphasub{\phi_d}{K}$.
Also,
$$
  \frac{\csub{E'_d}{F}}{\csub{E_d}{F}} = \frac{\csub{E'_d}{K}}{\csub{E_d}{K}}
$$ 
%both in the good and the split multiplicative case, by the semistable case of the theorem.
by the good and split multiplicative cases of the theorem.
Hence
$$
   \frac{\csub{E'\!}{K}}{\csub{E}{K}} =
   {\Bigl|\frac{E'_d(K)}{N E'_d(F)}\Bigr|}/{\Bigl|\frac{E_d(K)}{N E_d(F)}\Bigr|}.
$$
If $p=3$, then $\frac{\csub{E'\!}{K}}{\csub{E}{K}}$ is a power of 3
(Lemma \ref{tamprimetop})
but
the groups in the right-hand side are 2-groups, so $\csub{E'\!}{K}=\csub{E}{K}$. 

Finally, suppose $p=2$. 
If $E_d$ has good reduction (so $l=p$ and the reduction is good ordinary),
by \cite{KT} Prop~8.6
$\frac{E_d(K)}{N E_d(F)}$ has order 4 or 2, corresponding to
whether $\Delta_{E_d/K}$ is a norm in $F/K$ or not, and similarly for $E'$.
%%! Delta must be defined not up to units
%Noting that $\Delta/\Delta_{E_d/K}$ is a 6th power (as $E_d$ is a quadratic twist of $E$),
%we get the claim.
This gives the result for $c/c'$, noting that $\Delta_{E_d/K}$ can be
replaced by $\Delta_{E/K}$ in this criterion, since they differ by a 6th power
(Lemma \ref{discquadtwist}).

If $E_d$ has split multiplicative reduction, by \cite{KT} Prop~4.1
$\frac{E_d(K)}{N E_d(F)}$ has order 2 or 1, depending on whether
the Tate parameter $q$ of $E_d/K$ is a norm in $F/K$ or not.
Because $\Delta_{E_d/K}/q=\prod(1-q^n)^{24}$ is a square (\cite{Sil2} \S V.3)
and $\Delta_{E/K}/\Delta_{E_d/K}$ is a 6th power,
we get the same result as in the potentially ordinary case.
\end{proof}

%%%%%%%%%%%%%%%%%%%%%%%%%%%%%%%%%
\section{Real and complex periods}
%%%%%%%%%%%%%%%%%%%%%%%%%%%%%%%%%
\label{s:periods}

%\begin{lemma}
%Suppose $K$ is a finite extension of $\Q_l$, and $\omega, \omega'$ minimal
%differentials for $E/K$ and $E'/K$. Then
%$$
%   \frac{|\coker\phi: E(K)\to E'(K)|}{|\ker\phi: E(K)\to E'(K)|} = 
%      \alpha\frac{c'}{c},
%$$
%%
%%$$
%%  \frac{|\ker\phi: E(K)\to E'(K)|}{|\coker\phi: E(K)\to E'(K)|} = 
%%    .
%%$$
%If either $l\ne p$ or $E$ does not have potentially supersingular reduction,
%it is given by the following. If $E$ is semistable:
%\begin{itemize}
%\item 1 if $l\ne p$ and $E$ has potentially good reduction.
%\item $p^{\pm 1}$ if $l\ne p$ and $E$ has potentially multiplicative reduction.
%\item $1$ if $E$ has good reduction.
%\item $v_K\frac{\Delta(E)}{\Delta(E')}=-v_K\frac{j(E)}{j(E')}=p^{\pm 1}$ 
%if $E$ has split multiplicative reduction.
%\item if $E$ has nonsplit multiplicative reduction: $1$ if $p$ is odd;
%if $p=2$, then 
%$2^{\pm 1}$ if one of $v_K\Delta(E)$, $v_K\Delta(E')$ is odd, otherwise $1$.
%\end{itemize}
%If $E$ has additive reduction:
%\begin{itemize}
%\item $1$ if $p>3$.
%\item $p=3$, $l\ne 3$: $3^{\pm 1}$ if $E$ has type $\IV, \IVS$ and
%  $\mu_3\not\subset K$, otherwise $1$.
%\item $p=2$, $l\ne 2$: $2^{\pm 1}$ if $E$ has type $\IZS$ and
%  $\mu_4\not\subset K$, otherwise $1$.
%\item $1$ if $p=l=3$ and $E$ does not have additive potentially supersingular
% reduction.
%\item $w(q)$,... using KT if $p=l=2$ and $E$ does not have additive
%  potentially supersingular reduction.
%\end{itemize}
%\end{lemma}
%
%\begin{remark}
%Parity of its valuation we know in many cases.
%\end{remark}

\begin{notation}
In this section the field $\cK$ will be $\R$ or $\C$, and $\phi:E\to E'$
a $\cK$-rational $p$-isogeny of elliptic curves over $\cK$.
\end{notation}

\begin{definition}
The {\em period} of an elliptic curve $E/\cK$ %($\cK=\R$ or $\C$)
with respect
to an invariant differential $\omega$ is
$$
  \Omega(E,\omega) =
  {\int_{E(\cK)} |\omega|}\quad{\text{if $\cK\iso\R$}},
$$
and
$$
  \Omega(E,\omega) =
  {2\int_{E(\cK)} |\omega\wedge\bar\omega|}\quad{\text{if $\cK\iso\C$}}.
$$
\end{definition}

\begin{remark}
For $\cK=\R$, one sometimes uses the period $\Omega^+(E,\omega)$, which is obtained
by integrating only over the connected component of $E(\R)$. Thus
$\Omega\!=\!\Omega^+$ or $2\Omega^+$, depending on whether or
not $E(\R)$ is connected.
When working over $\Q$, one usually takes $\omega$ to be the global minimal 
differential and omits it from the notation.
\end{remark}

%\begingroup
%\setlength{\labelsep}{4pt}
%\setlength{\leftmargini}{18pt}
%\begin{remark}
%\par\noindent
%\begin{enumerate}
%\item[(a)]
%For $\cK=\R$, one sometimes uses the period $\Omega^+(E,\omega)$, which is obtained
%by integrating only over the connected component of $E(\R)$. Thus
%$\Omega\!=\!\Omega^+$ or $2\Omega^+$, depending on whether or
%not $E(\R)$ is connected.
%\item[(b)]
%When working over $\Q$, one usually takes $\omega$ to be the global minimal 
%differential and omits it from the notation.
%\item[(c)]
%There are various normalisations for complex periods. Our period 
%is purely imaginary, and it is its absolute value that enters 
%the formulation of the  Birch--Swinnerton-Dyer conjecture over number fields.
%\end{enumerate}
%\end{remark}
%\endgroup

\begin{lemma}
\label{archom}
%Let $\cK=\R$ or $\C$.
The periods of $E$ and $E'$ satisfy
$$
  \frac{\Omega(E,\omega)}{\Omega(E',\omega')} =
   \frac{|\ker\phi: E(\cK)\to E'(\cK)|}{|\coker\phi: E(\cK)\to E'(\cK)|} 
   \cdot \Bigl|\frac{\omega}{\phi^*\omega'}\Bigr|_\cK.
$$
\end{lemma}

\begin{proof}
The map $\phi: E(\cK)\to E'(\cK)$ is an $n$-to-1 unramified cover of 
$\phi(E(\cK))$, with $n=|\ker\phi|$. Therefore, if $\cK=\R$, then
$$
  \int_{E(\cK)} |\phi^*\omega'| = n \int_{\phi(E(\cK))} |\omega'|
    = \frac{n}{[E'(\cK): \phi(E(\cK))]} \int_{E'(\cK)} |\omega'|,
$$
and similarly for $\omega\wedge\bar\omega$ when $\cK=\C$. Hence
$$
  \frac{\Omega(E,\omega)}{\Omega(E',\omega')} =
  \frac{\Omega(E,\phi^*\omega')}{\Omega(E',\omega')} \cdot 
    \frac{\Omega(E,\omega)}{\Omega(E,\phi^*\omega')} = 
   \frac{|\ker\phi: E(\cK)\to E'(\cK)|}{|\coker\phi: E(\cK)\to E'(\cK)|} 
   \cdot \Bigl|\frac{\omega}{\phi^*\omega'}\Bigr|_\cK.
$$
\end{proof}

\begin{remark}
\label{remper}
If $K$ is an $l$-adic field with residue field $k$, then 
$$
  \int_{E(K)} |\omega|_K = \csub{E}{K} \frac{|E(k)|}{|k|},
$$
and $\frac{|E(k)|}{|k|}$ is the value of the Euler factor of $E$ at $s=1$.
These local integrals enter Tate's formulation of the
Birch--Swinnerton-Dyer conjecture~\cite{TatC}. Lemma \ref{archom} 
is the Archimedean analogue of Lemma \ref{alphalead}(2).
\end{remark}

%\pagebreak

\begin{proposition}
\label{kercokerR}
The quotient $\frac{|\ker\phi: E(\cK)\to E'(\cK)|}{|\coker\phi: E(\cK)\to E'(\cK)|}$ is
\begin{itemize}
\item $p$ if $\cK=\C$,
%\item $\Bigl\{\begin{tabular}{lll}
%      {p} {if $E[\phi]\subset E(\cK)$}\cr
%      {1} {if $E[\phi]\not\subset E(\cK)$}\cr
%    \end{tabular}\Bigr\}$ if $\cK=\R$ and $p\ne 2$.
\item $p$ if $\cK=\R$, $p\ne 2$ and $\ker\phi\subset E(\cK)$,
\item $1$ if $\cK=\R$, $p\ne 2$ and $\ker\phi\not\subset E(\cK)$.
\end{itemize}
If $\cK\!=\!\R$ and $p\!=\!2$, write $E$ in the form $y^2\!=\!x^3\!+\!ax^2\!+\!bx$ 
so that $(0,0)\in\ker\phi$. In this case, the quotient is
\begin{itemize}
\item 1 if $b>0$, and either $a<0$ or $4b>a^2$,
\item 2 otherwise.
\end{itemize}
\end{proposition}

%
%\begin{lemma}
%If $\cK=\C$, then
%$$
%  \frac{|\ker\phi: E(\cK)\to E'(\cK)|}{|\coker\phi: E(\cK)\to E'(\cK)|} = p.
%$$
%\end{lemma}
%
%\begin{lemma}
%If $\cK=\R$ and $p\ne 2$, then
%$$
%  \frac{|\ker\phi: E(\cK)\to E'(\cK)|}{|\coker\phi: E(\cK)\to E'(\cK)|} = 
%  \left\{
%    \begin{tabular}{lll}
%      {p}&&{$E[\phi]\subset E(\cK)$}\cr
%      {1}&&{$E[\phi]\not\subset E(\cK)$}\cr
%    \end{tabular}
%  \right.    
%$$
%\end{lemma}
%
%\begin{lemma}
%If $\cK=\R$ and $p=2$, write $E$ in the form $y^2=x^3+ax^2+bx$ 
%so that $(0,0)\in\ker\phi$. Then
%$$
%  \frac{|\ker\phi: E(\cK)\to E'(\cK)|}{|\coker\phi: E(\cK)\to E'(\cK)|} = 
%  \left\{
%    \begin{tabular}{lll}
%      {1}&&{$b>0$, and either $a<0$ or $4b>a^2$}\cr
%      {2}&&{otherwise}\cr
%    \end{tabular}
%  \right.    
%$$
%\end{lemma}

\begin{proof}
When $p\ne 2$, the cokernel is trivial and the result follows immediately.
For $p=2$, the kernel has size 2, and the cokernel is computed in 
\cite{isogroot} \S7.1.
\end{proof}

%%%%%%%%%%%%%%%%%%%%%%%%%%%%%%%%%
\section{Periods of elliptic curves over $\Q$}
%%%%%%%%%%%%%%%%%%%%%%%%%%%%%%%%%
\label{s:perQ}

Finally, we turn to periods of elliptic curves over $\Q$.

%Let $\phi: E\to E'$ be a $p$-isogeny of elliptic curves over $\Q$.
%We will write $\omega, \omega'$ for the global minimal differentials of $E$
%and $E'$, $\Omega=\Omega(E/\R,\omega)$, $\Omega'=\Omega(E'/\R,\omega')$ for their
%real periods, and $\Omega_\C = \Omega(E/\C, \omega)$, $\Omega'_\C = \Omega(E'/\C, \omega')$
%for their complex periods.

%Note that $\phi^*\omega'$ is either $\omega$ or $p\omega$, up to a sign.

\begin{notation}
For an elliptic curve $E$ over $\Q$ we write $\omega$ for the global
minimal differential on $E$ and
$$
 \Omega = \Omega(E/\R,\omega), \qquad \qquad \Omega_\C=\Omega(E/\C,\omega)
$$
for its real and complex periods. We similarly use $\omega',\Omega'$ and
$\Omega'_\C$ for $E'/\Q$.
%? Note that if $\phi:E\to E'$ is a rational $p$-isogeny, then $\phi^*\omega'$
%  is either $\omega$ or $p\omega$, up to a sign.
\end{notation}

\begin{theorem}
\label{perQpar}
Let $\phi\!:\!E\to E'$ be a rational $p$-isogeny of elliptic curves~over~$\Q$.
Then the quotient $\Omega/\Omega'$ is $p, 1$ or $p^{-1}$, and
the following are equivalent:
\begin{itemize}
\item[(1)]
$\Omega=\Omega'$.
\item[(2)] $\sum_l \ord_p(\frac{\csub{E}{\,\Q_l}}{\csub{E'\!}{\,\Q_l}}) \equiv \ord_{s=1}L(E,s)\mod 2$.
\end{itemize}
If $p\ne 2,3$, this is also equivalent to
\begin{itemize}
\item[(3)]
$E$ has an odd number of primes of additive
reduction with local root number $-1$.
%
%\noindent ($p=2$) There is an odd number of primes where either $E$ has additive
%reduction with local root number $-1$, or one of $E$ and $E'$ has
%non-split multiplicative reduction of type $\In{n}$ with $n$ odd.
\end{itemize}
The quotient $\Omega_\C/\Omega'_\C$ is $p$ or $p^{-1}$, and
it is $p$ if and only if $\omega=\pm\phi^*\omega'$.
%$
%  \frac{\Omega_\C}{\Omega'_\C} = p \quad\liff\quad \omega=\phi^*\omega'.
%$
\end{theorem}
\begin{proof}
By Lemma \ref{archom},
$$
  \frac{\Omega}{\Omega'} =
   \frac{|\ker\phi: E(\R)\to E'(\R)|}{|\coker\phi: E(\R)\to E'(\R)|} 
   \cdot \Bigl|\frac{\omega}{\phi^*\omega'}\Bigr|.
$$
The first term is $p$ or $1$ by Proposition \ref{kercokerR},
and the second term $|\frac{\omega}{\phi^*\omega'}|$ is either 1 or $p^{-1}$.
So $\Omega/\Omega'\in\{1,p,p^{-1}\}$. By the same argument over $\C$,
$|\frac{\omega}{\phi^*\omega'}|_\C$ is either 1 or $p^{-2}$, which immediately
gives the claim for the complex periods.

It remains to prove the equivalence of (1), (2) and (3). Now 
$\Omega=\Omega'$ if and only if $\ord_p(\frac{\Omega}{\Omega'})$ is even,
%and we can express this parity in terms of the local root numbers:
and we can relate this to the parity of the analytic rank and of 
the $p^\infty$-Selmer rank of $E/\Q$:
$$
  \ord_{s=1}L(E,s)  \equiv \rk_p E/\Q \equiv 
%  \displaystyle 
%  \qquad\equiv\ord_p|\frac{\omega}{\phi^*\omega'}| + 
%   \ord_p\frac{|\ker\phi: E(\R)\to E'(\R)|}{|\coker\phi: E(\R)\to E'(\R)|}      
%     + \sum_l \ord_p(\frac{\csub{E}{\Q_l}}{\csub{E'\!}{\Q_l}})\mod 2
%  \displaystyle\qquad
  \ord_p\frac{\Omega}{\Omega'} + \sum_l \ord_p(\frac{\csub{E}{\,\Q_l}}{\csub{E'\!}{\,\Q_l}})
%  \Bigl(\equiv \ord_p\frac{C(E)}{C(E')} \Bigr)
  \mod 2,
$$
by the $p$-parity conjecture over $\Q$ (\cite{Squarity} Thm. 1.4) 
and Cassels' formula for the parity of the $p^\infty$-Selmer rank 
for an elliptic curve with a $p$-isogeny (\cite{Squarity} Rmk. 4.4).
%This right-hand side is clearly
%$$
%  \ord_p\frac{\Omega}{\Omega'} + \sum_l \ord_p(\frac{\csub{E}{\Q_l}}{\csub{E'\!}{\Q_l}})
%  \Bigl(\equiv \ord_p\frac{C(E)}{C(E')} \Bigr),
%$$
This proves the equivalence $(1)\iff (2)$.

For $(2)\iff (3)$, suppose $p\ge 5$. Then 
by Theorem \ref{tammain}, 
%$$
%  (-1)^{\ord_p(\frac{\csub{E}{\Q_l}}{\csub{E'\!}{\Q_l}})} = \left\{ \begin{array}{lll}
%    1&\text{if $E$ has good reduction at $l$}&\quad(\csub{E}{\Q_l}=\csub{E'\!}{\Q_l}=1)\cr
%    -1&\text{if $E$ has split multiplicative reduction at $l$}&\quad\text{(see ?)}\cr
%    1&\text{if $E$ has non-split multiplicative reduction at $l$}&\quad (\csub{E}{\Q_l},\csub{E'\!}{\Q_l}\le 2)\cr
%    1&\text{if $E$ has additive reduction at $l$}&\quad (\csub{E}{\Q_l},\csub{E'\!}{\Q_l}\le 4)\cr
%  \end{array}
%  \right.
%$$
$\ord_p(\frac{\csub{E}{\,\Q_l}}{\csub{E'\!}{\,\Q_l}})$ is odd if and only if 
$E$ has split multiplicative reduction at $l$.
So the left-hand side in (2) is the number of primes of split multiplicative
reduction. The right-hand side is determined by the global root number
$w$ of $E$: it is even if $w=+1$ and odd if $w=-1$. Because $w$ is the product
of local root numbers,
$$
  w= - \prod\nolimits_l w_l 
$$ 
and the local root numbers $w_l=\pm 1$ are $-1$ for primes of split multiplicative reduction
and $+1$ for primes of good and nonsplit multiplicative reduction, the result
follows.
\end{proof}

%\begin{remark}
%Perhaps it is possible to formulate and prove this over arbitrary
%number fields, with $\rk_p E$ in the left-hand side. We only need to make
%sense of $\Omega$ and $\Omega'$ over number fields.
%
%\end{remark}

\begin{lemma}
\label{sernotsup}
Suppose $K/\Q_p$ is unramified. There are no elliptic curves over $K$ with
good supersingular reduction that admit a $p$-isogeny.
\end{lemma}

\begin{proof}
If $E$ has good supersingular reduction and $K/\Q_p$ is unramified,
Serre \cite{SerP} Prop. 12 proves that the image of Galois in $\Aut E[p]$ 
contains the non-split Cartan subgroup of $\GL_2(\F_p)$.
In particular, it acts irreducibly on $E[p]$, so $E/K$ cannot have a $p$-isogeny.
\end{proof}

%$p>2$. Little argument with formal groups and ramification of $K(E[p])/K$?
%See V's 4th term report essay?

\begin{lemma}
\label{pullmin}
Suppose $K/\Q_p$ is unramified, $p$ is odd, $E/K$ is semistable and 
$\omega, \omega'$ are minimal differentials on $E$ and $E'$. 
If $\ker\phi\subset E(K)$, then $\frac{\omega}{\phi^*\omega'}$ is a $p$-adic unit. 
%Otherwise, it is $p\cdot$unit. 
%
%If $\phi: E\to E'$ is a $p$-isogeny of semistable elliptic curves over $\Q_p$,
%with $p$ odd, $|\ker\phi|\ne 1$, then $\phi^*\omega'=\pm\omega$.
\end{lemma}

\begin{proof}
The curve $E$ cannot have supersingular reduction (Lemma \ref{sernotsup})
or non-split multiplicative reduction (as $\ker\phi\iso\Z/p\Z$ or $\mu_p$
in the split multiplicative case, and it has no points after an unramified 
quadratic twist). 
In the good ordinary case, Proposition \ref{omegaord} gives the claim. 
In the split multiplicative case, see the proof of Proposition \ref{omegapotmult}.
%
%Old proof:
%
%We want to show that $\phi^*\omega'$ is minimal. 
%It is minimal outside $p$, since
%$(\phi\phi^t)^*\omega'=p\omega'$ is.
%Also, $\frac{\omega}{\phi^*\omega'}$ is a $p$-adic unit 
%by the argument in \cite{isogroot} Lemma 14 `$\Leftarrow$'. 
%(If it had positive valuation, the formal group would acquire points over some
%extension of $\Q_p$. However, it does not, because all of $\ker\phi$ already
%lives in $\Q_p$ and either on the reduced curve or the component group, and
%so not on the formal group over any extension.)
\end{proof}

\begin{lemma}
\label{Omegassp}
If $\phi: E\to E'$ is a $p$-isogeny of elliptic curves over $\Q$,
with $p$ odd, $\ker\phi\subset E(\Q)$ and $E$ is semistable at $p$,
then $\frac{\Omega}{\Omega'}=\frac{\Omega_\C}{\Omega'_\C}=p$.
%$\Omega=p\Omega'$ and $\Omega_\C(E)=p\Omega_\C(E')$.
\end{lemma}

\begin{proof}
%One of $E, E'$ has a $p$-torsion point (fact), say $E$.
We have
$$
  \frac{\Omega}{\Omega'} \req={archom}
%  \frac{\Omega(E,\phi^*\omega')}{\Omega(E',\omega')} \cdot 
%    \frac{\Omega}{\Omega(E,\phi^*\omega')} = 
   \frac{|\ker\phi: E(\R)\to E'(\R)|}{|\coker\phi: E(\R)\to E'(\R)|} 
   \cdot \Bigl|\frac{\omega}{\phi^*\omega'}\Bigr|
      \req={kercokerR} 
   p\>\Bigl|\frac{\omega}{\phi^*\omega'}\Bigr| 
      \req={pullmin} p,
$$
and similarly for $\Omega_\C$.
\end{proof}

\begin{remark}
Note from the proof of the lemma that without the semistability assumption
for real periods we still have $\Omega\ge \Omega'$ 
when $\ker\phi\subset E(\Q)$.
\end{remark}

\begin{theorem}
\label{OmegaSS}
Let $\phi: E\to E'$ be a $p$-isogeny of semistable elliptic curves over $\Q$,
with $p$ odd. 
Then $\frac{\Omega}{\Omega'}=\frac{\Omega_\C}{\Omega'_\C}$ 
is either $p$ or $p^{-1}$.
%Then either $\Omega=p\Omega'$, $\Omega_\C=p\Omega'_\C$ or $p\Omega=\Omega'$, $p\Omega_\C=\Omega'_\C$.
Moreover,
$$
  \frac{\Omega}{\Omega'}=\frac{\Omega_\C}{\Omega'_\C}=p
     \quad\liff\quad
%  \Omega=p\Omega', \Omega_\C=p\Omega'_\C
%     \quad\liff\quad
  \frac{\omega}{\phi^*\omega'}=\pm 1    
     \quad\liff\quad
%  |E(\Q)[\phi]|=p
%     \quad\liff\quad
  \ker\phi\subset E(\Q).
$$
\end{theorem}

%$$
% \ord_p\frac{\Omega}{\Omega'} = \sum_l \ord_p(\frac{c_l}{c'_l})
%     + \rk_{an} E/\Q\mod 2
%$$
%

%\begin{remark}
%This also works over number fields, as long as $E$ is semistable at all primes
%above $p$, and $\ker\phi$ has points that do not lie on the formal group.
%\end{remark}

\begin{proof}
If $\ker\phi\subset E(\Q)$, then 
$\frac{\omega}{\phi^*\omega'}$ is a $p$-adic unit 
%\pagebreak
by Lemma \ref{pullmin} and unit at all other primes by Lemma \ref{lnep:omega},
so it is $\pm 1$. Also, $\frac{\Omega}{\Omega'}\!=\!\frac{\Omega_\C}{\Omega'_\C}\!=\!p$
by Lemma~\ref{Omegassp}.
If $\ker\phi\not\subset E(\Q)$, then by a result of Serre (\cite{SerP} p. 307),
$\ker\phi^t\subset E'(\Q)$. The result now follows from that for $\phi^t$, as
%$(\phi^t)^*\phi^*\omega'=p\omega'$
$$
  p\cdot
  \frac{\omega}{\phi^*\omega'}=
  \frac{\phi^*(\phi^t)^*\omega}{\phi^*\omega'}=
  \frac{(\phi^t)^*\omega}{\omega'}=\pm 1.
$$
\end{proof}

%\begin{corollary}
%\label{Omegass}
%If $\phi: E\to E'$ is a $p$-isogeny of semistable elliptic curves over $\Q$,
%with $p$ odd, then either $\Omega(E)=p\Omega(E')$ or $p\Omega(E)=\Omega(E')$.
%\end{corollary}

%%%%%%%%%%%%%%%%%%%%%%%%%%%%%%%%%
\section*{Appendix A. Tate curve and quadratic twists}
%%%%%%%%%%%%%%%%%%%%%%%%%%%%%%%%%
\def\thesection{A}
\refstepcounter{section}
\label{s:apptate}

For completeness, we recall the following well-known facts.
These concern the Tate curve and quadratic twists of elliptic curves, and
do not assume that $E$ admits a $p$-isogeny.
As usual, $K$ is a finite extension of $\Q_l$, and the notation is as in \S1.1.

\begin{theorem}
\label{tatecurve}
An elliptic curve $E/K$ with split multiplicative reduction of type $\In{n}$ 
is isomorphic to a Tate curve $E^{(q)}/K$ for some Tate parameter 
$q\in \m_K$ with $v(q)=n=\delta=-v(j)=c$. For a prime $p$,
$$
  E^{(q)}(\bar K) \iso \bar K^\times/q^\Z \qquad\text{and}\qquad
  E^{(q)}[p]\iso \langle \zeta_p, \sqrt[p]q\rangle
$$
as Galois modules. There is a $K$-rational $p$-isogeny
%If $q\notin K^{\times p}$, the unique $K$-rational $p$-isogeny
%from $E^{(q)}$ is 
\beq
  E^{(q)}(K) &=& K^\times/q^\Z &\lar& K^\times/q^{p\Z} &=& E^{(q^p)}(K) \cr
         &&      z         &\mapsto& z^p.\cr
\eeq
The other $p$-isogenies from $E^{(q)}$ are parametrised by choices of a $p$th root
of $q$ in $\bar K^\times$, and given by 
\beq
  E^{(q)}(K) &=& K^\times/q^\Z &\lar& K^\times/(\!\sqrt[p]q)^{\Z} &=& E^{(\sqrt[p]q)}(K) \cr
         &&      z         &\mapsto& z.\cr
\eeq
Such an isogeny is defined over $K$ if and only if $\sqrt[p]q\in K$.
%If $q\in K^{\times p}$ then for every $\alpha\in K^\times$ with $\alpha^p=q$
%(there is either 1 or $p$ such choices depending on whether or not $\mu_p\subset K^\times$,
%there is also a $K$-rational $p$-isogeny
\end{theorem}

\begin{proof}
For the basic theory of the Tate curve, see \cite{Sil2} \S V.3-V.5.
For the statements about isogenies, see \cite{SerA} \S A.1.4 
(Theorem and the proof of (2)$\implies$(1)), 
and the description of the function field of $E_q$ in \S A.1.1.
\end{proof}

\begin{lemma}
\label{quadmult}
If $E/K$ has potentially multiplicative reduction, both the 
quadratic twist of $E$ by $-c_6$ and $E/K(\sqrt{-c_6})$ have
split multiplicative reduction.
Here $c_6$ is the standard invariant of $E$ as in {\rm\cite{Sil1} \S III.1}. 
\end{lemma}

\begin{proof}
See \cite{Sil2} \S V.5.
\end{proof}

\begin{lemma}
\label{quadtwist}
Let $E/K: y^2=f(x)$ be an elliptic curve with additive 
potentially good reduction. Then the following are equivalent:
\begin{enumerate}
\item $E$ has good reduction over a 
  quadratic extension $K(\sqrt d)$.
\item The quadratic twist $E_d/K: dy^2=f(x)$ has good reduction.
\item The inertia group $\Gal(\bar K/K^{nr})$ acts on the $\ell$-adic 
Tate module of $E$ ($\ell\ne l$) through $\Gal(K^{nr}(\sqrt d)/K^{nr})$.
\end{enumerate}
%
% if and only if the 
%quadratic twist $E_d/K: dy^2=f(x)$ has semistable reduction. 
\end{lemma}

\begin{proof}
This follows from the criterion of N\'eron-Ogg-Shafarevich, and the fact
that the Tate module of $E_d$ is the Tate module module of $E$ twisted by
the character of $\Gal(K(\sqrt d)/K)$ of order 2.
\end{proof}

\begin{lemma}
\label{discquadtwist}
Let $E/K$ be an elliptic curve, and $E_d/K$ its quadratic twist by $d\in K^\times$.
Then the minimal discriminants of $E$ and $E_d$ are related by $\Delta_{E/K} = d^6 u^{12} \Delta_{E_d/K}$
for some $u\in K^\times$.
\end{lemma}

\begin{proof}
Choose models $E: y^2=x^3+ax+b$  and $E_d: y^2=x^3+d^2ax+d^3b$.
Their discriminants differ by $d^6$, and they differ from the 
minimal discriminants by 12th powers.
\end{proof}

\begin{theorem}[\cite{Lor} Thm 2.8]
\label{lorthm}
Suppose $E/K$ has multiplicative reduction of type $\In n$, so $n=\delta=-v(j)$.
Let $K(\sqrt d)/K$ be a quadratic extension, and $\chi:\Gal(K(\sqrt d)/K)\to\pm 1$ 
the corresponding character.
Then the quadratic twist $E_d$ has potentially multiplicative 
reduction of type $\InS{n+4f_\chi-4}$, where
$f_\chi$ is the conductor exponent of $\chi$. 
If $p\ne 2$, then $f_\chi=1$ and the type is $\InS n$.
\end{theorem}

%\begin{corollary}
%\label{lorcor}
%In the notation of the theorem, 
%$$
%  \delta_{E_\chi/K} = m_{\chi} + 2f_\chi - 1 = (5+n+4f_\chi-4) + 2f_\chi - 1 = n + 6f_\chi
%    = -v(j) + 6f_\chi,
%$$
%where $m_\chi$ is the number of connected components of the special fibre 
%of the N\'eron model of $E_\chi/K$.
%\end{corollary}

%%%%%%%%%%%%%%%%%%%%%%%%%%%%%%%%%
\section*{Appendix B. Values of modular forms}
%%%%%%%%%%%%%%%%%%%%%%%%%%%%%%%%%
\def\thesection{B}
\refstepcounter{section}
\label{s:appmod}

We briefly recall the well-known connection between values of modular forms on
$\Gamma_0(N)$ and invariants of elliptic curves with a cyclic $N$-isogeny. 
In the proof of Theorem~\ref{deltaquo}, we used the exact relation 
\eqref{deldel} between the discriminant of an elliptic curve and the value 
of the modular $\Delta$-function, and rationality properties of values
of modular forms. For the latter, the modern approach of Katz
goes via the $q$-expansion principle (see \cite{Katz} or \cite{DI}).
For convenience of the reader, we also give a low-tech description,
which relies only on classical results (Theorem \ref{modeval}).

%
%The modern approach of Katz is described in \cite{Katz} or \cite{DI}. We give a 
%low-tech description, 

Let $E/\C$ be an elliptic curve with an invariant differential $\omega$.
Put $(E,\omega)$ in the form
\begin{equation}
\label{E4eq}
  E: y^2 = 4x^3 + ax + b, \qquad \omega=\frac{dx}y.
\end{equation}
By the uniformisation theorem, there 
is a unique lattice $\Lambda=\Z\Omega_1+\Z\Omega_2\subset\C$ such that
$$
\begin{array}{cll}
  \phi: \C/\Lambda & \lar & E(\C) \cr
           z       & \longmapsto & (\wp_\Lambda(z),\wp'_\Lambda(z)) \cr
\end{array}
$$
is an isomorphism (of complex Lie groups), and $\phi^*\frac{dx}y=dz$.
Here $\wp_\Lambda(z)$ is the Weierstrass $\wp$-function
$$
  \wp_\Lambda(z) = \frac{1}{z^2} + \sum_{v\in\Lambda\setminus\{0\}} 
    \Bigl( 
    \frac{1}{(z-v)^2}-\frac{1}{v^2} \Bigr).
$$
The coefficients of $E$ are 
$a=-60G_4(\Lambda)$ and $b=-140G_6(\Lambda)$,
where the $G_k$ are the standard modular functions
$$
  G_k(\Lambda) = \sum_{v\in\Lambda\setminus\{0\}} v^{-k}.
$$
Let $\tau=\frac{\Omega_2}{\Omega_1}$,
changing the sign of $\Omega_2$ if necessary to get $\tau\in\bH$.
Write $\Lambda_\tau = \Z\tau+\Z$ and $q=e^{2\pi i \tau}$.
Then 
$$
  \Lambda=\Omega_1\Lambda_\tau, \qquad G_k(\Lambda)=\Omega_1^{-k} G_k(\Lambda_\tau),
$$
and $\tau\mapsto G_k(\Lambda_\tau)$ is, up to a constant, the 
Eisenstein series of weight $k$,
$$
  E_k(\tau) = \frac{1}{2\zeta(k)} G_k(\Lambda_\tau) = 
    1 + \frac{2}{\zeta(1-2k)} \sum_{n=1}^\infty \frac{n^{2k-1}q^n}{1-q^n}.
$$
The modular discriminant function $\Delta(\tau)=\frac{1}{1728}(E_4(\tau)^3-E_6(\tau)^2)$
satisfies
\begin{equation}
\label{deldel}
  \Delta(\tau) = 
    -16\Bigl(4(\tfrac{a}{4})^3+27(\tfrac{b}{4})^2\Bigr)
    \cdot \Bigl(\frac{\Omega_1}{2\pi}\Bigr)^{12} = 
      \Bigl(\frac{\Omega_1}{2\pi}\Bigr)^{12}\Delta_E,
\end{equation}
where $\Delta_E$ is the discriminant of 
the Weierstrass model $y^2 = x^3 + \tfrac a4 x+\tfrac b4$ of $E$
obtained by rescaling $y\mapsto 2y$ in \eqref{E4eq} 
(so $\frac{dx}{y}$ becomes $\frac{dx}{2y}$).

% a=-60 G4, b=-140 G6
% E4 = 1/2/zeta(4) G4 = (1/2/(pi^4/90)) G4 = 45/pi^4 G4
% E6 = 1/2/zeta(6) G6 = (1/2/(pi^6/945)) G4 = 945/2pi^6 G6
% E4^3 - E6^2 = Delta(tau)
% See also disctest.m

Now suppose that the pair $(E,\omega)$ is defined over a subfield $\cK\subset \C$. 
Then $a, b, \Delta\in \cK$, and therefore 
$$
  \bigl(\tfrac{2\pi}{\Omega_1}\bigr)^4 E_4(\tau), \quad
  \bigl(\tfrac{2\pi}{\Omega_1}\bigr)^6 E_6(\tau), \quad 
  \bigl(\tfrac{2\pi}{\Omega_1}\bigr)^{12} \Delta(\tau) \quad\in\>\> \cK.
$$
In fact, suppose $f(\tau)$ is any modular form on $\Gamma_0(N)$ whose 
$q$-expansion has $\cK$-rational coefficients. 
For any choice of non-negative integers $m,n,u$ with $4m+6n+k=12u$, the form
$\tilde f=f E_4^m E_6^n/\Delta^u$ on $\Gamma_0(N)$ has weight 0 (i.e. is 
a modular function). 
By a classical theorem (see \cite{Cox} Thm. 11.9(b)), 
$$
  \tilde f(\tau)=F(j(\tau),j(N\tau))
$$  
for some rational function $F\in\C(x,y)$. In fact, $F\in \cK(x,y)$ since 
$F$ has a $\cK$-rational $q$-expansion.%
\footnote{This is clear for rational functions
of $j(\tau)$, since the $q$-expansion of $j(\tau)$ is rational, hence Galois
invariant, and $\C(t)^{\Aut(\C/\cK)}=\cK(t)$; in general, write 
$\tilde f$ as a unique polynomial in $j(N\tau)$ with coefficients 
in $\C(j(\tau))$ of degree $<n$, where $n$ is the degree of the 
modular polynomial $\Phi_N(x,y)$ relating $j(\tau)$ and $j(N\tau)$, and apply
the same Galois invariance argument to it and its coefficients.}
Summarising the whole discussion, we have

\begin{theorem}
\label{modeval}
Let $f\in M_k(\Gamma_0(N))$ be a modular form whose $q$-expansion has 
$\cK$-rational coefficients, $\cK\subset \C$.
There are natural numbers $m,n,u$ and a rational function $F\in \cK(x,y)$ 
such that for every cyclic isogeny of elliptic curves 
$\phi: E\to E'$ of degree $N$, with $E$ in the form \eqref{E4eq} 
with corresponding complex
lattices $\Lambda=\Z\Omega_1+\Z\Omega_2\subset\C$,
$\Lambda'=\Z\Omega_1+\Z N\Omega_2\subset\C$, we have
$$
  \bigl(\tfrac{2\pi}{\Omega_1}\bigr)^k f(\tfrac{\Omega_2}{\Omega_1}) 
     = \frac{a^m b^n}{\Delta_E^u} F(j(E),j(E')).
$$
In particular, if $E$ and $E'$ are defined over $\cK$,
the left-hand side lies in~$\cK$.
\end{theorem}

%Finally, note that we can replace $\Q$ by any subfield of $\C$. 
%In fact, by embedding the field of definition of $E, E'$ and $\phi$ into $\C$
%we can get the result over any field $\cK$ of characteristic 0 
%(Lefschetz principle). 

%%%%%%%%%%%%%%%%%%%%%%%%%%%%%%%%%%%%%%%%%%%%%%


\begin{thebibliography}{99}

\bibitem{Bir}
B. J. Birch, Conjectures concerning elliptic curves,
Proc. Sympos. Pure Math., Vol. VIII (1965), Amer. Math. Soc., Providence,
R.I, 106--112.

\bibitem{Kes} K. \v Cesnavi\v cius, 
The $p$-parity conjecture for elliptic curves with a $p$-isogeny,
preprint, 2012, arxiv: 1207.0431.

\bibitem{CoaLMS}
J. Coates, Elliptic curves with complex multiplication and Iwasawa theory,
Bulletin of the LMS 23 (1991), 321--350.

\bibitem{Cox}
D. Cox, Primes of the Form $x^2 + ny^2$, John Wiley \& Sons, Inc., New York, 1989.

\bibitem{DI}
F. Diamond, J. Im, Modular forms and modular curves, 
in Seminar on Fermat's Last Theorem,  CMS Conf. Proc., 17, AMS,
Providence, RI, 1995, 39--133.

\bibitem{isogroot}
T. Dokchitser, V. Dokchitser,
Parity of ranks for elliptic curves with a cyclic isogeny,
J. Number Theory 128 (2008), 662--679.

\bibitem{Squarity}
T. Dokchitser, V. Dokchitser,
On the Birch--Swinnerton-Dyer quotients modulo squares,
Annals of Math. 172 no. 1 (2010), 567--596.

\bibitem{megasha}
T. Dokchitser, V. Dokchitser, Growth of $\smallsha$ 
in towers for isogenous curves, preprint, 2013, arxiv:1301.4257.

\bibitem{Fis}
T. Fisher, Appendix to V. Dokchitser,
Root numbers of non-abelian twists of elliptic curves,
Proc. London Math. Soc. (3) 91 (2005), 300--324.

\bibitem{GS} B. Gordon, D. Sinor, Multiplicative properties of 
$\eta$-products, 
Number theory, Madras 1987, LNM 1395, Springer, Berlin, 1989, pp. 173--200.

\bibitem{Katz}
N. Katz, $p$-adic properties of modular schemes and modular forms, in
Modular functions of one variable III, Springer-Verlag LNM 350 (1973),
69--190.

\bibitem{Keil} S. Keil, Examples of abelian surfaces with non-square
Tate-Shafarevich group, preprint, 2012, arxiv: 1206.1822. 

\bibitem{Klya}
A. Klyachko, Modular forms and representations of symmetric groups,
{\em Integral lattices and finite linear groups}, Zap. Nauchn. Sem. 
LOMI {\bf 116}, Nauka, 1982, 74--85.

\bibitem{KT}
K. Kramer, J. Tunnell, Elliptic curves and local $\epsilon$-factors,
Compositio Math. 46 (1982), 307--352.

\bibitem{Kra}
A. Kraus, Sur le d\'efaut de semi-stabilit\'e des courbes elliptiques
\`a r\'eduction additive, Manuscripta Math. 69 (1990), no. 4, 353--385.

\bibitem{Lor}
D. Lorenzini, Models of curves and wild ramification, Pure and
Appl. Math. Quart. 6 (2010), 41--82.

\bibitem{New}
M. Newman, Construction and application of a class of modular 
functions (2), Proc. London Math. Soc. (3), 9 (1959), 373--387.

\bibitem{Scha}
E. Schaefer, Class groups and Selmer groups,
J. Number Theory 56 (1996), no. 1, 79--114.

\bibitem{SerA}
J.-P. Serre, Abelian $l$-adic Representations and Elliptic Curves,
Addison-Wesley, Reading, 1989.

\bibitem{SerP}
J-P. Serre, Propri\'et\'es galoisiennes des points d'ordre
fini des courbes elliptiques, Invent. Math. {\bf 15} (1972), 259--331.

\bibitem{ST}
J.-P. Serre, J. Tate, Good reduction of abelian varieties, Annals of Math. 
(2) 88 (1968), 492-–517.

\bibitem{Sil1}
J. H. Silverman, The Arithmetic of Elliptic Curves,
Graduate Texts in Mathematics 106, Springer-Verlag 1986.

\bibitem{Sil2}
J. H. Silverman, Advanced Topics in the Arithmetic of Elliptic Curves,
Graduate Texts in Mathematics 151, Springer-Verlag 1994.

\bibitem{TatC}
J. Tate, On the conjectures of Birch and Swinnerton-Dyer and a
geometric analogue, S\'eminaire Bourbaki, 18e ann\'ee, 1965/66, no. 306.



\end{thebibliography}
\end{document}